\documentclass[12pt]{amsart}
\usepackage{amsfonts}
\usepackage{amsmath}
\usepackage{amssymb}
\usepackage{graphicx}
\usepackage{enumerate}
\usepackage{multicol}
\usepackage{color}
\usepackage[text={5.35in,8.3in}]{geometry}

\newcommand{\real}{{\mathbb{R}}}

\newtheorem{theorem}{Theorem}
\newtheorem{corollary}{Corollary}
\newtheorem{lemma}{Lemma}
\newtheorem{definition}{Definition}
\newtheorem{proposition}{Proposition}
\theoremstyle{remark}
\newtheorem{remark}{Remark}
\newtheorem{example}{Example}

\DeclareMathOperator{\im}{Image }
\DeclareMathOperator{\SO}{SO}
\DeclareMathOperator{\GL}{GL}
\DeclareMathOperator{\spn}{span}

\DeclareMathOperator{\Isom}{Isom}
\DeclareMathOperator{\pr}{pr}
\DeclareMathOperator{\SE}{SE}
\DeclareMathOperator{\se}{\mathfrak{se}}

\DeclareMathOperator{\tr}{trace}

\title[rolling manifolds]{An intrinsic formulation of the rolling manifolds problem}
\author[M. Godoy, E. Grong, I. Markina, F. Silva]{Mauricio Godoy Molina\\Erlend Grong\\Irina Markina\\F\'atima Silva Leite}

\address{Department of Mathematics, University of Bergen, Norway.}
\email{mauricio.godoy@math.uib.no}

\address{Department of Mathematics, University of Bergen, Norway.}
\email{erlend.grong@math.uib.no}

\address{Department of Mathematics, University of Bergen, Norway.}
\email{irina.markina@uib.no}

\address{Department of Mathematics and Institute of Systems and Robotics, University of Coimbra, Portugal.}
\email{fleite@mat.uc.pt}

\thanks{The first three authors are supported by the grant of the
Norwegian Research Council \# 177355/V30 and by the grant of the
European Science Foundation Networking Programme HCAA. The last author is partially supported by FCT under project PTDC/EEA-ACR/67020/2006.}



\begin{document}

\maketitle

\begin{abstract}

We present an intrinsic formulation of the kinematic problem of two $n-$dimensional manifolds rolling one on another without twisting or slipping. We determine the configuration space of the system, which is an $\frac{n(n+3)}2-$dimensional manifold. The conditions of no-twisting and no-slipping are decoded by means of a distribution of rank $n$. We compare the intrinsic point of view versus the extrinsic one. We also show that the kinematic system of rolling the $n$-dimensional sphere over $\mathbb R^n$ is controllable. In contrast with this, we show that in the case of $SE(3)$ rolling over $\mathfrak{se}(3)$ the system is not controllable, since the configuration space of dimension 27 is foliated by submanifolds of dimension 12.

\end{abstract}

\section{Introduction}

Rolling of surfaces without slipping or twisting is one of the classical kinematic problems that in recent years has again attracted the attention of mathematicians due to its geometric and analytic richness. The kinematic conditions of rolling without slipping or twisting are described by means of motion on a configuration space being tangential to a smooth sub-bundle that we call a distribution. The precise definition of the mentioned motion in the case of two $n$-dimensional manifolds imbedded in $\real^m$, given for example in~\cite{Sharpe}, involves studying the behavior of the tangent bundles of the manifolds and the normal bundles induced by the imbeddings. This approach leads to significant simplifications, for instance, it suffices to study the case in which the still manifold is the $n-$dimensional Euclidean space. The drawback is that the geometric descriptions depend strongly on the imbedding under consideration.

However, so far little attempts have been made to formulate this problem intrinsically. An early enlightening formulation is given in~\cite{BH}, in which the authors study the case of two abstract surfaces rolling in the above described manner. This is achieved by means of an intrinsic version of the moving frame method of \'Elie Cartan which, for this case, coincides with the classical intrinsic study of surfaces, see~\cite{ST}. One of important results established in~\cite{BH} is the non-integrability property of a rank two distribution corresponding to no-twisting and no-slipping restrictions, namely, if the two surfaces have different Gaussian curvature, then the distribution is completely non-integrable and moreover is of Cartan-type, see~\cite{Cartan2}. A control theoretic approach to the same problem, studied in~\cite{AS}, has the advantage that the kinematic restrictions are written explicitly as vector fields on appropriate bundles.

We present a generalization of the kinematic problem for two $n-$dimensional abstract manifolds rolling without twisting or slipping via an intrinsic formulation. We define the configuration space of the system, which is an $\frac{n(n+3)}2-$dimensional manifold and which is a direct analogue to the one found in the references~\cite{BH} and~\cite{AS}. We give several equivalent definitions of rolling motion involving intrinsic characteristics and those that depend only on imbedding and discuss their relations. This new definitions permit to determine the imbedding-independent information contained in the extrinsic definition of the rolling bodies problem presented in~\cite{Sharpe}. Moreover, we relax the smoothness condition of the rolling map up to absolutely continuity. This allows to enlarge the class of mappings under consideration, still giving the possibility to apply the fundamental theorems of differential geometry and control theory without changing drastically the main classical ideas of rolling maps. The conditions of no-twisting and no-slipping define a distribution of rank $n$ in the tangent bundle of the configuration space. We write explicitly the distribution as a local span of vector fields defined on the configuration space. We test the bracket generating condition of the above mentioned distribution on the known example~\cite{Zimm} of rolling the $n$-dimensional sphere over the $n$-dimensional Euclidean space and the special group of Euclidean rigid motions $SE(3)$ rolling over $\mathfrak{se}(3)$. As a result we obtain the controllability of the first system and the non controllability of the latter.

The structure of the present paper is the following. Section~\ref{Def} is an introductory section where we collect necessary definitions and discuss the motivation for the reformulations of kinematic conditions of no-twisting and no-slipping for the rolling problem. We present two formulations and show their equivalence. Section 3 gives a good starting point for comparing different approaches, known in the literature for 2-dimensional rolling manifolds. In Section 4 we give the main formulation of extrinsic rolling as a curve on a configuration space defined as a direct sum of principal bundles over Cartesian product of two rolling manifolds and we prove the equivalence of new extrinsic definition of rolling with the previous ones and deduce the {\it intrinsic} definition of a rolling map. We also prove some theorem enlightening the imbedding independent information contained in the principal definition of the extrinsic rolling. Section 5 is dedicated to the construction of two distributions in the tangent bundle of the configuration space. These distributions encode the no-twisting and no-slipping kinematic conditions of extrinsic and intrinsic rolling. The rolling, both extrinsic and intrinsic, can be written as a curves in configuration space tangent to the corresponding distributions. In Sections 6 and 7 we present detailed calculations for the two aforementioned examples: rolling the $n$-dimensional sphere over the $n$-dimensional Euclidean space and rolling $SE(3)$ over $\mathfrak{se}(3)$. In the first case the distribution is bracket generating, coinciding with the result obtained in~\cite{Zimm}. In the second case we obtain that the configuration space, of dimension 27, is foliated by 12 dimensional submanifolds.

\section{Definition of rolling map for manifolds imbedded in Euclidean space}\label{Def}

\subsection{Rolling without twisting or slipping for imbedded manifolds}
We start from the classical definition of rolling without slipping or twisting of one manifold over another manifold inside the Euclidean space.

Let us start with some notations. Throughout this paper, $M$ and $\widehat{M}$ will always be oriented connected Riemannian manifolds of dimension $n$. By the well known result of Nash, see \cite{Nash}, there are isometric imbeddings of $M$ and $\widehat{M}$, denoted by $\iota$ and $\widehat{\iota}$ respectively, into $\real^{n+\nu}$ for an appropriate choice of~$\nu$. Here and in what follows $\real^{n+\nu}$ will always be equipped with the standard Euclidean metric and standard orientation. As long as there is no possibility for confusion, we will identify the abstract manifolds $M$ and $\widehat{M}$ with their images under the corresponding imbeddings. The imbedding of $M$ into $\real^{n+\nu}$ splits the tangent space of $\real^{n+\nu}$ into a direct sum:
\begin{equation} \label{split}
T_x\real^{n+\nu} = T_xM \oplus T_xM^\perp, \quad x \in M.
\end{equation}
In general, any objects (points, curves, \dots) related to the manifold $\widehat{M}$ will be marked by a hat ($\,\widehat{\,\,}\,$) on top, objects related to $M$ will be free of it, while terms related to the ambient $\mathbb R^{n+\nu}$ space carry a bar ($\,^{-} \,$). We use $\Isom(M)$ for the group of isometries of $M$, and $\Isom^+(M)$ for the group of sense preserving isometries.

We start by given the definition of rolling without twisting and slipping as found in~\cite{Sharpe}.

\noindent{\bf Definition 0.} {\it Let $M, \widehat{M}$ be submanifolds of $\mathbb{R}^{n+\nu}$.
Then, a differentiable map $g:[0,\tau] \to \Isom(\mathbb R^{n+\nu})$ satisfying the following conditions for any
$t \in [0,\tau]$ is called a rolling $M$ on $\widehat{M}$ without slipping or twisting.
The rolling conditions:
\begin{itemize}
\item there is a piecewise smooth curve $x:[0,\tau] \to M$, such that
\begin{itemize}
\item $g(t) x(t) \in \widehat{M}$,
\item $T_{g(t)x(t)}\left(g(t) M \right) = T_{g(t)x(t)} \widehat{M}$.
\end{itemize}
\item Furthermore, the curve $\widehat{x}(t) := g(t) x(t)$ satisfies the following
\begin{itemize}
\item no-slip condition: $$\dot{g}(t) g(t)^{-1} \widehat{x}(t) = 0,$$
\item no-twist condition, tangential part:
$$d(\dot{g}(t) g(t)^{-1}) T_{\widehat{x}(t)}\widehat M \subseteq T_{0}(\dot{g}(t) g(t)^{-1}\widehat M)^\bot,$$
\item No-twist condition, normal part:
$$d(\dot{g}(t) g(t)^{-1}) T_{\widehat{x}(t)}\widehat M^\bot \subseteq T_{0}(\dot{g}(t) g(t)^{-1}\widehat M)$$
\end{itemize}
\end{itemize}
}

\begin{remark}
In the previous definition, we explicitly state that $g:[0,\tau] \to \Isom(\mathbb R^{n+\nu})$ is differentiable. This is not stated in~\cite{Sharpe}, but conditions containing $\dot{g}$ are required to hold for all $t$. Also, a minor inaccuracy in the no-twisting conditions is corrected.
\end{remark}

It is clear that Definition 0 is of extrinsic nature. Thus, in order to obtain an intrinsic formulation of the rolling problem, we want to change the original definition as follows:

\smallskip

\noindent (1) {\it Making $x(t)$ part of the data of the rolling:} The reason is to give a local character to conditions of rolling without twisting or slipping. This will emphasize the dependence of the rolling not just on the isometry $g$ but also on a curve $x$ along which the rolling of $M$ on $\widehat{M}$ can be realized. In some particular cases, this may lead to small changes in terminology. The following example illustrates these ideas.
\begin{example}
Consider the submanifolds of $\mathbb{R}^3$, defined by
$$M = \left\{ (\bar{x}_1, 1 - \cos \theta, \sin \theta) \in \mathbb{R}^3 \vert\
\, \bar{x}_1 \in \mathbb{R},\  \theta \in [0,2\pi) \right\},$$
$$\widehat{M} = \left\{ (\bar{x}_1, \bar{x}_2, 0) \in \mathbb{R}^3 \vert\ \, \bar{x}_1, \bar{x}_2 \in \mathbb{R},\right\}.$$
The rolling map
$$g(t): \bar{x} = \left(\begin{array}{c} \bar{x}_1 \\ \bar{x}_2 \\ \bar{x}_3 \end{array} \right)
\mapsto \left(\begin{array}{c} \bar{x}_1 \\ (\bar{x}_2 -1) \cos t + \bar{x}_3 \sin t + t + \cos 2t \\
- (\bar{x}_2 -1) \sin t + \bar{x}_3 \cos t - \sin 2t \end{array} \right),$$
describes the rolling of the infinite cylinder $M$ on $\widehat M$ along the $\bar{x}_2$-axis
with constant speed 1.
Then there is an infinite choice of curves $x(t)\in M$, given by $$x(t) = (\bar{x}_1, 1- \cos t, \sin t),\quad \bar{x}_1
\in \mathbb{R}$$ along which the rolling $g$ can be realized. However, if we make $x(t)$ as part of the data, then each choice of the curve $x(t)$ will correspond to different rollings $\big(x(t),g(t)\big)$.
\end{example}

\smallskip

\noindent  (2) {\it Relaxing the differentiability conditions for $g(t)$:} We think that the conditions of differentiability of $g(t)$ for all $t\in[0,\tau]$ and piecewise smoothness of $x(t)$ are too restrictive. The requirement that $(x,g):[0, \tau] \to M\times \Isom(\real^{n+\nu})$ is absolutely continuous or Lipschitz seems more natural, since this allows us to implement results from control theory, see Subsection~\ref{ASsubsection}. In this context, absolute continuity of a curve $\big(x(t),g(t)\big)$ on $M\times \Isom(\mathbb R^{n+\nu})$ is considered with respect to the parameter $t$, as in \cite[Chapter 2]{AS}.

\smallskip

\noindent (3) {\it Introducing orientability assumptions:} In order to have a connected configuration space, we exploit the orientability assumption of $M$ and $\widehat{M}$. Since, as mentioned before, the rolling conditions will be local, we may choose an orientable neighborhood of the starting point even on any non-orientable manifold. We will use this to impose some practical restrictions to the definition of a rolling.
\begin{itemize}
\item Since $g(t)$ is continuous, it is either always orientation preserving or orientation reversing isometry of $\mathbb{R}^{n+\nu}$ for all $t$. Given a rolling  $g(t)$ of $M$ on $\widehat{M}$, we may assume that $g(t)$ is always orientation preserving by changing the orientation of~$\mathbb R^{n+\nu}$. To obtain an orientation preserving rolling from an orientation reversing rolling $g(t)$ of $M$ on $\widehat{M}$, pick any constant orientation reversing isometry $g_0$ of $\mathbb R^{n+\nu}$. Then $g_0 g(t)$ is an orientation preserving rolling of $M$ on $g_0(\widehat{M})$.

\item It is intuitively clear that for a fixed $t$, $d_{x(t)}g(t)$ maps elements from $T_{x(t)}M$ to $T_{\widehat{x}(t)}\widehat{M}$ and elements from $T_{x(t)}M^\bot$ to $T_{\widehat{x}(t)}\widehat{M}^\bot$  (for more details see Subsection~\ref{firstreform}). Hence, the matrix form of $d_{x(t)} g(t)$ splits in the following way:

$$\begin{array}{cccc}
    & & \begin{array}{cc} T_{x(t)} M & T_{x(t)} M^\bot \end{array} & \\
    d_{x(t)} g(t) & = & \left(\begin{array}{cc} \, \, A(t) \,\, & \,\, 0 \,\, \\
    \, \, 0 \,\, & \, \, B(t) \,\, \end{array}\right) &
    \begin{array}{l}
    T_{\widehat{x}(t)} \widehat{M} \\
    T_{\widehat{x}(t)} \widehat{M}^\bot.
    \end{array}
\end{array}$$
Since $g(t)$ is orientation preserving, both $d_{x(t)} g(t)|_{T_{x(t)}M}$ and $d_{x(t)} g(t)|_{T_{x(t)}M^\perp}$ are either orientation preserving or orientation reversing. By continuity, $d_{x(t)} g(t)|_{T_{x(t)}M}$ is either orientation preserving or orientation reversing for all $t$. We will require that $d_{x(t)} g(t)|_{T_{x(t)}M}$ is always orientation preserving. If $d_{x(t)} g(t)|_{T_{x(t)}M}$ is orientation reversing, pick any constant orientation preserving isometry $g_0:\mathbb R^{n+\nu}\to\mathbb R^{n+\nu}$
so that
$$dg_0|_{T\widehat{M}}: T\widehat{M} \to T(g_0\widehat{M})$$
is orientation reversing. It is sufficient to show that it reverses the orientation at one point in order to show that it reverses orientation at all points due to the fact that $M$ is oriented. Then $g_0g(t)$ will be a rolling of $M$ on $g_0\widehat{M}$ which is orientation preserving on the tangent space at $x(t)$.
\end{itemize}

Implementing the above changes to Definition 0, we obtain the following, from which several equivalent reformulations will be presented later.
\begin{definition} \label{imbeddef}
A rolling of $M$ on $\widehat{M}$ without twisting or slipping is an absolutely continuous curve $(x,g):[0,\tau] \rightarrow M \times \Isom^+(\mathbb{R}^{n + \nu})$, satisfying the following conditions:
\begin{itemize}
\item[(i)] $\widehat{x}(t) := g(t) x(t) \in \widehat{M}$, for all $t\in[0,\tau]$.
\item[(ii)] $T_{\widehat{x}(t)}( g(t) M) = T_{\widehat{x}(t)} \widehat{M}$, for all $t\in[0,\tau]$.
\item[(iii)] No slip condition: $\dot{g}(t) \circ g^{-1}(t) \widehat{x}(t) = 0,$
for almost every $t$.
\item[(iv)] No twist condition (tangential part):
$$d(\dot{g}(t) \circ g^{-1}(t)) (T_{\widehat{x}(t)} \widehat{M}) \subseteq T_0 (\dot{g}(t) \circ g^{-1}(t) \widehat{M})^\perp,$$ for almost every $t$.
\item[(v)] No twist condition (normal part):
$$d(\dot{g}(t) \circ g^{-1}(t)) (T_{\widehat{x}(t)} \widehat{M}^\perp) \subseteq T_0 (\dot{g}(t) \circ g^{-1}(t) \widehat{M}),$$ for almost every $t$.
\item[(vi)] $d_{x(t)} g(t)|_{T_{x(t)} M}:T_{x(t)} M \to T_{\widehat{x}(t)} \widehat{M}$ is orientation preserving, for all $t\in[0,\tau]$.
\end{itemize}
\end{definition}

We omit, from now on, the words ``without twisting or slipping'', just writing ``a rolling of $M$ on $\widehat{M}$''. Furthermore, for given curves $x(t)$ and $\widehat{x}(t)$ in $M$ and $\widehat{M}$, respectively, the expression "a rolling of $M$ on $\widehat{M}$ along $x(t)$ and $\widehat{x}(t)$" will mean a rolling $(x,g):[0,\tau] \rightarrow M \times \Isom^+(\mathbb{R}^{n+\nu})$ so that $g(t)x(t) = \widehat{x}(t)$.

\begin{remark}
The definitions we will be working with ignore physical restrictions given by the actual shapes of the manifolds. Intuitively, if we think of the manifolds in Definition \ref{imbeddef} as physically touching along the curves $x(t)$ and $\hat{x}(t)$ and rolling according to the isometry $g(t)$, then we cannot rule out the possibility that there might be non-tangential intersections between the manifolds other than the contact points.
\end{remark}

\begin{example}
Consider the imbedded surface
$$M=\{(x_1,x_2,x_3) \in \mathbb{R}^3 \big| \, x_1^2 - x_2^2 + x_3 = 0, \, x_1^2 + x_2^2 < 1 \},$$
and $\widehat{M}=\mathbb{R}^2,$ imbedded as an affine plane. Assume that both manifolds $M$ and $\widehat{M}$ carried the induced metric. We can clearly define a rolling of $M$ on $\widehat{M}$ in terms of Definition \ref{imbeddef}, but there is no way to connect the saddle point in $M$ with any point in $\widehat{M}$ without there being intersections between the surfaces.
\end{example}


\subsection{First reformulation}\label{firstreform}
We aim to give a definition of the rolling of $M$ on $\widehat{M}$ in a way that is more fruitful for future considerations. We fix some notations first. According to the splitting \eqref{split}, any vector $v \in T_x\real^{n+\nu},$ $x \in M$, can be written uniquely as the sum $v = v^\top+v^\bot$, where $v^\top\in T_xM$ is tangent to $M$ at $x$, while $v^\bot\in T_xM^\perp$ is normal. Analogous projections can be defined for $\widehat{M}$.

Let $\nabla$ denote the Levi-Civita connection on $M$ or on $\widehat{M}$. The context will indicate on which manifold the connection is defined. The ``ambient'' Levi-Civita connection on $\mathbb{R}^{n+\nu}$ is denoted by $\overline{\nabla}$. Note that if $X$ and $Y$ are tangent vector fields on $M$, then
$$\nabla_X Y(x) = \left(\overline{\nabla}_{\bar{X}} \bar{Y}(x) \right)^\top, \qquad x \in M,$$
where $\bar{X}$ and $\bar{Y}$ are any local extensions to $\mathbb{R}^{n+\nu}$ of the vector fields $X$ and $Y$, respectively. Similarly, if $\Upsilon$ is a normal vector field on $M$ and $X$ is a tangent vector field on $M$, then the normal connection is defined by
$$\nabla_X^\bot \Upsilon(x) = \left(\overline{\nabla}_{\bar{X}} \bar{\Upsilon}(x) \right)^\bot, \qquad x \in M,$$
where $\bar{\Upsilon}$ is any local extension to $\mathbb{R}^{n+\nu}$ of the vector field $\Upsilon$. Equivalent statements hold for $\widehat{M}$. If no confusions arise, we will use capital Latin letters $X,Y,Z$ to denote tangent vector fields and capital Greek letters $\Upsilon,\Psi$ for notation of normal vector fields.

For a fixed value of $x \in M$ and
a fixed vector field $Y$, the vector $\nabla_X Y(x)$ only depends on the value of $X(x)$.
Therefore, for $v \in T_xM$, we will use $\nabla_v Y$ or $\nabla_v Y(x)$
to mean $\nabla_XY(x)$, where $X$ is an arbitrary vector field satisfying $X(x) = v$.
We will use the same convention when $\nabla$ is interchanged with $\nabla^\bot$.

If $Z(t)$ is a vector field along $x(t)$, we will use $\frac{D}{dt} Z(t)$ to denote the covariant derivative
(corresponding to $\nabla$) of $Z(t)$ along $x(t)$, and for any normal vector field $\Psi(t)$ along $x(t)$,
$\frac{D^\bot}{dt} \Psi(t)$ denotes the normal covariant derivative (see~\cite[p. 119]{ONeill}).
Recall that if $M$ is imbedded isometrically into $\mathbb{R}^{n+\nu}$, then
$$\frac{D}{dt} Z(t) = \left(\frac{d}{dt} Z(t) \right)^\top, \qquad
\frac{D^\bot}{dt} \Psi(t) = \left(\frac{d}{dt} \Psi(t) \right)^\bot,$$
where $Z(t)$ and $\Psi(t)$ are tangential and normal vector fields, respectively, along a curve
in $M$.

We say that a tangent vector $Y(t)$ along an absolutely continuous curve $x(t)$ is parallel if $\frac{D}{dt} Z(t) = 0$
for almost every $t$. Notice that it is possible to define the notion of parallel transport even though the derivative $\dot{x}(t)$ exists only almost everywhere, see,~e.~g., Existence and Uniqueness Theorem in~\cite[Appendix C]{Son}. Namely, for any absolutely continuous curve $x:[0,\tau] \to M$ and for any $v \in T_{x(t_0)}M$, $0 \leq t_0 \leq \tau$, there exists a unique absolutely continuous tangent vector field $Z(t)$ along $x(t)$, such that $Z(t)$ is parallel and satisfies $Z(t_0) =v$.

We say that a normal vector field $\Psi(t)$ along $x(t)$ is {\it normal parallel} if $\frac{D^\bot}{dt}(t) \Psi=0$
for almost every $t$. A normal analogue of parallel transport is defined likewise.

We are now ready to give a new formulation of the rolling map.
\begin{definition} \label{imbeddef2}
A rolling of $M$ on $\widehat{M}$ without slipping or twisting is an absolutely continuous curve $(x,g):~[0,\tau] \to M \times \Isom^+(\real^{n+\nu})$ satisfying the following conditions:
\begin{itemize}
\item[(i')] $\widehat{x}(t) := g(t) x(t) \in \widehat{M}$,
\item[(ii')] $dg(t) T_{x(t)}M = T_{\widehat{x}(t)} \widehat{M}$,
\item[(iii')] No slip condition: $\dot{\widehat{x}}(t)= dg(t) \dot{x}(t),$ for almost every $t$.
\item[(iv')] No twist condition (tangential part):
$$dg(t) \frac{D}{dt} Z(t) = \frac{D}{dt} dg(t) Z,$$
for any tangent vector field $Z(t)$ along $x(t)$ and almost every $t$.
\item[(v')] No twist condition (normal part):
$$dg(t)\frac{D^\bot}{dt} \Psi(t) = \frac{D^\bot}{dt} dg(t) \Psi(t),$$
for any normal vector field $\Psi(t)$ along $x(t)$ and almost every $t$.
\item[(vi')] $d_{x(t)} g(t)|_{T_{x(t)} M}:T_{x(t)} M \to T_{\widehat{x}(t)} \widehat{M}$ is orientation preserving.
\end{itemize}
\end{definition}

\begin{lemma} \label{equiv1}
Definitions \ref{imbeddef} and \ref{imbeddef2} are equivalent.
\end{lemma}

\begin{proof}
Since (i) and (i') coincide, we begin by proving the equivalence of (ii) and (ii'). Restricting the action of $g(t)$ to $M$, we observe that the differential $d_{x(t)}g(t)$ maps $T_{x(t)}M$ into $T_{g(t)x(t)} \left(g(t) M \right)$ by definition, and hence (ii) holds if and only if (ii') holds.

In order to prove the equivalence between (iii) and (iii') we write a curve $g(t)$ in $\Isom^+(\real^{n+\nu})$ as
$$g(t): \bar{x} \mapsto \bar{A}(t) \bar{x} + \bar{r}(t), \qquad \bar{x} \in \real^{n+\nu},$$
where $\bar{A}:[0,\tau] \to \SO(n+\nu)$ and $\bar{r}:[0,\tau] \to \real^{n+\nu}$. Thus $d_{\bar{x}} g(t) v = \bar{A}(t) v$, $v \in T_{\bar{x}}\real^{n+\nu}$, and we get
\begin{align}
\dot{g}(t) \circ g^{-1}(t) \, \widehat{x}(t) & = \dot{g}(t) x(t) = \dot{\bar{A}}(t) x(t) + \dot{\bar{r}}(t) \nonumber \\
& =\frac{d}{dt}\left(\bar{A}(t) x(t) + \bar{r}(t) \right)  -\bar{A}(t) \dot{x}(t) = \dot{\widehat{x}}(t) - dg(t) \dot{x}(t). \nonumber
\end{align}
whenever $\dot{x}(t)$ is defined. Hence $\dot{g}(t) \circ g^{-1}(t) \widehat{x}(t)= 0$ almost everywhere if and only if $\dot{\widehat{x}}(t) = dg(t) \dot{x}(t)$ almost everywhere.

Before we continue with the final two conditions, notice that (ii') actually states that both $dg(t)(T_{x(t)}M) = T_{\widehat{x}(t)} \widehat{M}$ and $dg(t)(T_{x(t)}M^\bot) = T_{\widehat{x}(t)} \widehat{M}^\bot$ hold due to the splitting~\eqref{split}. Hence, the inverse differential $dg^{-1}(t) = (dg(t))^{-1}$ also maps tangent vectors to tangent vectors and normal vectors to normal vectors. This allows us to restate (iv) and (v) as the conditions
$$\left(d\dot{g}(t) v^\top \right)^\top = 0, \quad \text{and} \quad \left(d\dot{g}(t) v^\bot \right)^\bot = 0,$$
holding for almost every $t$ and for any $v\in T_{x(t)}\mathbb{R}^{n+\nu}$, decomposed as the sum of $v^\top \in T_{x(t)}M$ and $v^\bot \in T_{x(t)}M^\bot$ via the splitting~\eqref{split}.
We calculate
\begin{eqnarray}
0 &=& \left( d\dot{g}(t) \, Z(t) \right)^\top =  \left( \frac{d}{dt} \big(dg(t) \, Z(t)\big) - dg(t) \left(\frac{d}{dt} Z(t) \right) \right)^\top \nonumber\\
&=& \frac{D}{dt} dg(t) \, Z(t) - dg \frac{D}{dt} Z(t)\nonumber
\end{eqnarray} for any tangent vector field $Z(t)$ along $x(t)$, for any value of $t$ where $\dot{x}(t)$ is defined.
By similar calculations, using a normal vector field $\Psi(t)$ along $x(t)$, we obtain
$$dg(t)\frac{D^\bot}{dt} \Psi(t) = \frac{D^\bot}{dt} dg(t) \, \Psi(t).$$
\end{proof}

\begin{remark}
The following observations are useful for the understanding of the nature of a rolling map.
\begin{itemize}
\item The proof of Lemma~\ref{equiv1} shows that indeed condition (ii') is equivalent to the statement $$dg(t) T_{x(t)}M^\perp = T_{\widehat{x}(t)}\widehat{M}^\perp.$$
\item Condition (iv') is equivalent to the requirement that any tangent vector field $Z(t)$ is parallel along $x(t)$ if and only if $dg(t) Z(t)$ is parallel along $\widehat{x}(t)$. As a consequence, this condition is automatically satisfied in the case of one dimensional manifolds.
\item We can reformulate (v') in terms of normal parallel vector fields. Namely, condition (v') is equivalent to the statement that any normal vector field $\Psi(t)$ is normal parallel along $x(t)$ if and only if $dg(t) \Psi(t)$ is normal parallel vector field along $\widehat{x}(t)$. Thus, if the manifolds are imbedded into Euclidean space and the codimension is one (i.e. $\nu =1$), condition (v') always holds.
\end{itemize}
\end{remark}

\section{Previous intrinsic descriptions of rolling maps dimension 2}

The aim of this Section is to present the different intrinsic formulations of a rolling map appearing in literature for two dimensional manifolds. The two best known formulations are given in~\cite{AS,BH}. We start by introducing the configuration space of the rolling for the general case of $n$ dimensional manifolds and then proceed to describe the previously mentioned two dimensional situation.

\subsection{Frame bundles and bundles of isometries} \label{conspace}

Let $M$ and $\widehat M$ be oriented connected Riemannian manifolds of dimension $n$. We introduce the configuration space $Q$ of the rolling, which can be thought of as all the relative positions in which $M$ can be tangent to $\widehat{M}$. Define the principal $\SO(n)$-bundle over $M \times \widehat{M}$ by
$$Q = \left\{\left. q \in \Isom^+_0(T_xM,T_{\widehat{x}} \widehat{M}) \, \right| \, x \in M, \, \widehat{x} \in \widehat{M} \right\}.$$
Here $\Isom_0^+(V,\widehat{V})$ denotes the group of linear orientation preserving isometries between the oriented inner product spaces $V$ and $\widehat{V}$.

The principal $\SO(n)$-bundle structure of the configuration space $Q$ can be also described in the following way. Let $F$ and $\widehat{F}$ be oriented frame bundles of $M$ and $\widehat{M}$, respectively, with the obvious principal $\SO(n)$-bundle structures. Consider $F \times \widehat{F}$ as a bundle over $M \times \widehat{M}$ with $\SO(n)$ acting diagonally on the fibers. Then, we can identify $Q$ with
$\left(F \times \widehat{F}\right)/\SO(n)$ by the following map. Let $\{e_j(x)\}_{j=1}^n$ be a frame in $F$ at $x\in M$ and similarly let $\{\hat{e}_i(\widehat{x})\}_{i=1}^n$ be a frame in $\widehat{F}$ at $\widehat{x}\in\widehat M$. To the equivalence class $$\left(\{e_j(x)\}_{j=1}^n, \{\hat{e}_i(\widehat{x})\}_{i=1}^n \right)\cdot \SO(n)$$ we
associate the unique isometry $q \in \Isom_0^+(T_xM, T_{\widehat{x}} \widehat{M})$ satisfying
\begin{equation} \label{QandFFcor}
\hat{e}_i(\widehat{x}) = q \, e_i(x), \qquad i = 1,\dots n.
\end{equation}
Clearly, this construction does not depend on the choice of a representative of an equivalence class of
$\left(F\times \widehat{F}\right)/\SO(n)$. Conversely, given an isometry, there exists a unique equivalence class of frames satisfying~\eqref{QandFFcor}.

The left and right action on fibers of $Q$ is induced by the inverse left action on $F$ and left action on $\widehat{F}$, respectively. More precisely, an element $A_0 \in \SO(n)$ acts on an isometry $q\in Q$ from the right or left, giving $q_1 = qA_0$ and $q_2 = A_0 q$, respectively. The isometries $q_1$ and $q_2$ are defined by
$$\hat{e}_i(\widehat{x}) = q_1 \, A_0^{-1} e_i(x), \qquad A_0\hat{e}_i(\widehat{x}) = q_2 \, e_i(x),$$ where
$\left(\{e_j(x)\}_{j=1}^n, \{\hat{e}_i(\widehat{x})\}_{i=1}^n \right)$ is any basis satisfying \eqref{QandFFcor}.

Furthermore, let $\left(\{f_j(x)\}_{j=1}^n, \{\hat{f}_i(\widehat{x})\}_{i=1}^n \right) \in (F \times \widehat{F})|_{(x,\widehat{x})}$ be any other pair of frames and the matrix representation $A = \left( a_{ij} \right)_{i,j=1}^n$, of an isometry $q$ is given by
$$A = \left( a_{ij} \right)_{i,j=1}^n= \Big(\hat{f}_i^*(\widehat{x})\, q \, f_j(x) \Big)_{i,j=1}^n := \left(\Big\langle \hat{f}_i(\widehat{x}), q \, f_j(x) \Big\rangle \right)_{i,j=1}^n \in \SO(n),$$ where $\hat{f}_i^*$ stands for the 1-form dual to the vector field $\hat{f}_i$.
Then, in this basis,
$$\Big(\hat{f}_i^*(\widehat{x})\, q_1 \, f_j(x) \Big)_{i,j=1}^n = AA_0, \quad \text{and} \quad \Big(\hat{f}_i^*(\widehat{x})\, q_2 \, f_j(x) \Big)_{i,j=1}^n = A_0A.$$

Since $Q$ is a principal $\SO(n)$-bundle over $M \times \widehat{M}$, it has dimension $\frac{n(n+3)}{2}$ as a manifold.

\subsection{Agrachev-Sachkov formulation of rolling surfaces} \label{ASsubsection}
A previous definition of a rolling map can be found in \cite{AS}, where only 2-dimensional manifolds imbedded into $\mathbb{R}^3$ are considered. Although it only deals with the imbedded case, the definition of the rolling is intrinsic in the sense that it does not depend on the imbedding.

The configuration space for rolling one surface on another is $Q$, which is now 5-dimensional, since $M$ and $\widehat{M}$ are 2-dimensional. A rolling is then an absolutely continuous curve $q:[0,\tau] \to Q$ satisfying the following: if $x(t)$ and $\widehat{x}(t)$ are the projections of $q(t)$ into $M$ and $\widehat{M}$ then the following two conditions are satisfied:
\begin{itemize}
\item no slip condition: $\dot{\widehat{x}} = q(t) \, \dot{x}(t)$ for almost every $t\in[0,\tau]$;
\item no twist condition: $Z(t)$ is a parallel tangent vector field along $x(t)$ if and only if $q(t) Z(t)$ is a parallel tangent vector field along $\widehat{x}(t)$.
\end{itemize}
Notice that there is no condition corresponding to the normal no-twist, since the manifolds here have codimension 1. In Section \ref{introll} we will show how this definition fits into our Definition~\ref{imbeddef2}.

The no-slip and no-twist conditions can be described by means of a distribution $D$ in the tangent bundle of~$Q$. By distribution, we mean a smooth subbundle of the tangent bundle. Then the ``no slip -- no twist'' condition will correspond to the requirement $\dot{q}(t) \in D_{q(t)}$ for almost every $t$. The distribution $D$ has the following local description. In any sufficiently small neighborhood $U\subset M$ of $y\in M$ we pick a pair of tangent vector field $e_1, e_2$, such that $\{e_1(x), e_2(x) \}$ is a positively oriented orthonormal basis for every $x \in U$. Define $\hat{e}_1, \hat{e}_2$ in a similar way in a sufficiently small neighborhood $\widehat{U}$. Since the rotation group $SO(2)$ has dimension 1, we simply need to know the relative angle $\theta$ to describe $q$ with respect to the frames given by $\{e_1, e_2\}$ and $\{\hat{e}_1, \hat{e}_2\}$. More precisely, $\theta$ is defined by
\begin{eqnarray}
q \, e_1 & = & \cos \theta \hat{e}_1 + \sin \theta \hat{e}_2,\nonumber\\
q \, e_2 & = & -\sin \theta \hat{e}_1 + \cos \theta \hat{e}_2.\nonumber
\end{eqnarray}
Thus, if $\pi:Q \to M\times \widehat{M}$ is the natural projection, then any $q \in \pi^{-1}(U \times \widehat{U})$ is uniquely determined by the coordinates $(x,\widehat{x},\theta), (x,\widehat{x}) \in U \times \widehat{U}$.

Let $c_1,c_2, \widehat{c}_1$ and $\widehat{c}_2$ be the so-called ``structural constants'', defined by the commutation relations
$$[e_1,e_2] = c_1 e_1 + c_2 e_2, \qquad [\hat{e}_1,\hat{e}_2] = \widehat{c}_1 \hat{e}_1 + \widehat{c}_2 \hat{e}_2.$$
Define the vector fields $X_1$ and $X_2$ on $\pi^{-1}(U \times \widehat{U})$ by
\begin{equation} \label{X1AS}
\begin{split}X_1 = e_1 + \cos\theta \hat{e}_1 + \sin\theta \hat{e}_2 + (-c_1 + \widehat{c}_1 \cos\theta + \widehat{c}_2 \sin\theta) \frac{\partial}{\partial \theta} \, , \\
X_2 = e_2 - \sin\theta \hat{e}_1 + \cos\theta \hat{e}_2 + (-c_2 - \widehat{c}_1 \sin\theta + \widehat{c}_2 \cos\theta) \frac{\partial}{\partial \theta} \, .
\end{split}
\end{equation}
Then $D|_{\pi^{-1}(U \times \widehat{U})}$ is spanned by $X_1,X_2$.

The connectivity by a curve tangent to the distribution $D$ is the principal problem. More precisely, given two different states $q_0,q_1 \in Q$, we ask whether there exists a rolling motion $q:[0,\tau] \rightarrow Q$, such that $q(0) = q_0$ and $q(\tau) = q_1$? The advantage of the formulation of no slipping and no twisting conditions in terms of a distribution, is that the question of connectivity may be reformulated through admissible sets or orbits in control theory.

Given a distribution $D$ on an arbitrary manifold $Q$, a curve $q:[0,\tau] \to Q$ is said to be {\it horizontal} (or {\it admissible}) with respect to $D$ if $q$ is an absolutely continuous curve satisfying $\dot{q}(t) \in D$ for almost every $t$. The {\it orbit} of $D$ at a point $q_0$ is the set of all points $q_1 \in Q$ so that there exists a curve $q:[0,\tau] \to Q$, with $q(0) = q_0$ and $q(\tau) = q_1$, which is horizontal with respect to $D$. We denote this set by $\mathcal{O}_{q_0}(D)$. It is clear that if $q_1 \in \mathcal{O}_{q_0}(D)$, then $\mathcal{O}_{q_0}(D) = \mathcal{O}_{q_1}(D)$. The Orbit Theorem \cite{Hermann, Suss} asserts that $\mathcal{O}_{q_0}$ is an immersed submanifold of $Q$ and describes the tangent space of the orbit in terms of
the diffeomorphisms of $Q$. A precise statement using the chronological exponential and a broad discussion
about the Orbit Theorem can be found in Chapter 5 of \cite{AS}.

Also, define the flag associated to the
distribution $D$ inductively by
$$D^1 = D \text{ and } D^{i+1} = D + [D,D^i].$$
We say that $D$ has step $k \geq 2$  at $q$ if $k$ is the maximal integer, so that
$$D^{k-1}_q \subsetneq D_q^k = D_q^{k+1}.$$
If $D^k_q = D_q$ for any integer $k$, we say that $D$ has step 1 at $q$.
The Orbit Theorem then tells us that $D^k_q \subseteq T_q\mathcal{O}_{q_0}(D)$,
where $k$ is the step at $q \in \mathcal{O}_{q_o}(D)$.
In particular, if $Q$ is connected and there is an integer $k$ such that $D^k=TQ$, then $\mathcal{O}_{q_0}(D) = Q$. The previous result is known as the Chow-Rashevski{\u\i} theorem \cite{Chow,Rashevsky} and the distribution $D$ is called bracket generating.

We will use the expression that $D$ has step $k$ if $D$ has step $k$
for any $q \in Q$. Remark that if $D$ is of step $k$, and there is a local basis of vector fields of $D^k$
in a neighborhood around any point in $Q$, then
$$D^k_q = T_q\mathcal{O}_{q_0}(D).$$

We now go back to the intrinsic definition presented in~\cite{AS}, where $Q$ is the described 5-dimensional configuration space and $D$ is spanned locally by \eqref{X1AS}. This definition can be restated as following: a curve $q(t)\colon [0,\tau]\to Q$ is the rolling map of $M$ on $\widehat M$ if it is tangent to $D$. The main result of \cite{AS}, is the following description of orbits of $D$. Let $\varkappa(x)$ and $\widehat{\varkappa}(\widehat{x})$ denote the Gaussian curvature of $M$ at $x$ and of $\widehat{M}$ at $\widehat{x}$, respectively.

\begin{theorem}
For any $q_0 \in Q$, the orbit at $q_0$ satisfies $\dim \mathcal{O}_{q_0}(D) = 2$ if and only if $\varkappa(\pr_M q) - \widehat{\varkappa}(\pr_{\widehat{M}} q) = 0$ for every $q \in \mathcal{O}_{q_0}(D)$. Otherwise, $\dim \mathcal{O}_{q_0}(D) = 5$.
\end{theorem}

\begin{remark}
In contrast to the definition in~\cite{Sharpe}, the definition in~\cite{AS} deals with absolutely continuous curves. The advantage of this, is the ability to apply the Orbit theorem and the Chow-Rashevski{\u\i} theorem. This was one of the reasons for us to define a rolling map in terms of absolutely continuous curves. Remark that all these theorems also hold if we consider Lipschitz curves instead of absolutely continuous. Hence, we always may interchange ``absolutely continuous'' with ``Lipschitz'' for all considerations in the present paper.
\end{remark}

\subsection{Bryant-Hsu formulation of rolling surfaces}\label{BHsection}
In~\cite{BH} the authors give an intrinsic formulation to the problem of rolling two abstract surfaces $M$ and $\widehat{M}$ with respect to each other. The main tool in this formulation is Cartan's general method of moving frames, that is, determining canonical forms on an appropriate $\SO(2)-$bundle.

Let $M$ and $\widehat{M}$ be two connected oriented Riemannian manifolds of dimension 2. Consider the respective frame bundles $F$, $\widehat{F}$. Then, as discussed in Subsection~\ref{conspace}, the configuration space $Q$ for this kinematic system can be identified with $(F \times \widehat{F})/\SO(2)$. The conditions of no twisting and no slipping can be understood by means of the canonical one-forms $\alpha_1,\alpha_2$, $\alpha_{21}$ on $F$ and $\widehat\alpha_1,\widehat\alpha_2$, $\widehat\alpha_{21}$ on $\widehat{F}$. Recall, that these forms satisfy the structure equations
\vspace{-0.7cm}
\begin{multicols}{2}
\begin{eqnarray}
d\alpha_1&=&\alpha_{21}\wedge\alpha_2,\nonumber\\
d\alpha_2&=&-\alpha_{21}\wedge\alpha_1,\nonumber\\
d\alpha_{21}&=&\varkappa\;\alpha_{1}\wedge\alpha_2,\nonumber
\end{eqnarray}

\begin{eqnarray}
d\widehat\alpha_1&=&\widehat\alpha_{21}\wedge\widehat\alpha_2,\nonumber\\
d\widehat\alpha_2&=&-\widehat\alpha_{21}\wedge\widehat\alpha_1,\nonumber\\
d\widehat\alpha_{21}&=&\widehat\varkappa\;\widehat\alpha_{1}\wedge\widehat\alpha_2,\nonumber
\end{eqnarray}
\end{multicols}
\noindent where $\varkappa$ and $\widehat\varkappa$ are the Gauss curvatures of $M$ and $\widehat{M}$ respectively, see \cite[Chapter 7]{ST}.

The rank two distribution $D$ over $Q$ corresponding to the ``no slip -- no twist'' conditions is the push-forward of the vector fields, solving the Pfaffian equations
\begin{equation}\label{2Dcase}
\alpha_1-\widehat\alpha_1=\alpha_2-\widehat\alpha_2=\alpha_{21}-\widehat\alpha_{21}=0,
\end{equation}
under the natural projection $\pi:F\times \widehat{F} \to Q$. At the points where $\varkappa-\widehat\varkappa\neq 0$ the distribution $D$ is of Cartan type, that is, the distributions
$$D^2=D+[D,D]\quad\mbox{and}\quad D^3=D^2+[D,D^2]$$
have rank 3 and 5 respectively, see~\cite{BH}. This implies that, under the condition $\varkappa-\widehat\varkappa\neq 0$, the distribution $D$ is bracket generating of step 3. To see under which conditions $D$ is of Cartan type, define the following one-forms over the product $F\times\widehat{F}$
\begin{eqnarray*}
\theta_1=\frac12(\alpha_1-\widehat\alpha_1), & \theta_2=\dfrac12(\alpha_2-\widehat\alpha_2), & \theta_3=\frac12(\alpha_{21}-\widehat\alpha_{21}),
\end{eqnarray*}
$$\omega_1=\frac{1}{2}(\alpha_1+\widehat\alpha_1),\quad \omega_2=\frac{1}{2}(\alpha_2+\widehat\alpha_2),$$
and observe that the following identities hold:
\begin{eqnarray}
d\theta_1&=&\theta_3\wedge\omega_2+\frac12(\alpha_{21}+\widehat\alpha_{21})\wedge\theta_2, \nonumber \\
d\theta_2&=&-\theta_3\wedge\omega_1-\frac12(\alpha_{21}+\widehat\alpha_{21})\wedge\theta_1, \nonumber \\
d\theta_3&=&\frac12(\varkappa-\widehat\varkappa)\omega_1\wedge\omega_2+\frac12\Big((\varkappa+\widehat\varkappa)(\omega_1\wedge\theta_2-\omega_2\wedge\theta_1)+(\varkappa-\widehat\varkappa)\theta_1\wedge\theta_2\Big). \nonumber
\end{eqnarray}

Denote by $\mathcal D=\ker\theta_1\cap\ker\theta_2\cap\ker\theta_3$ the space of solutions of the system \eqref{2Dcase} and let $X=(X_1,X_2),Y=(Y_1,Y_2),Z=(Z_1,Z_2)$ be a local basis of $\mathcal D$ chosen such that
$$\begin{array}{cccc}
\alpha_1(X_1)=1,\quad&\alpha_2(X_1)=0,\quad&\widehat\alpha_1(X_2)=1,\quad&\widehat\alpha_2(X_2)=0,\\
\alpha_1(Y_1)=0,\quad&\alpha_2(Y_1)=1,\quad&\widehat\alpha_1(Y_2)=0,\quad&\widehat\alpha_2(Y_2)=1,\\
\alpha_1(Z_1)=0,\quad&\alpha_2(Z_1)=0,\quad&\widehat\alpha_1(Z_2)=0,\quad&\widehat\alpha_2(Z_2)=0.
\end{array}$$

Observe that for a sufficiently small open neighborhood $U\times\widehat U\subset M\times\widehat M$ of $(p,\widehat p)$, the differential of the projection $\pi$ is
$$\begin{array}{ccccc}d_{((p,C),(\widehat p,\widehat C))}\pi&:&T_pU\times\mathfrak{so}(2)\times T_{\widehat p}\widehat U\times\mathfrak{so}(2)&\to&T_pU\times T_{\widehat p}\widehat U\times \mathfrak{so}(2)\\&&(x,A,y,B)&\mapsto&(x,y,A-B)\end{array}$$
for any $C,\widehat C\in\SO(2)$ and where $T_C\SO(2)$, $T_{\widehat C}\SO(2)$ and $T_{C\widehat C^{-1}}\SO(2)$ are identified with $\mathfrak{so}(2)$ in the usual manner. By the construction of the canonical forms on the frame bundles, it is clear that $X,Y\notin\ker d\pi$, whereas it is possible to choose locally $Z$ such that $Z\in\ker d\pi$. Thus since $\ker d\pi$ has dimension one, we have locally
$$\ker d\pi=\spn\{Z\}.$$
This implies that a local description of $D$ is given by
$$D=\spn\{d\pi(X),d\pi(Y)\}.$$

Recall Cartan's formula for a differential one form $\eta$ and any two local vector fields $v,w$, given by
\begin{equation*}\label{cartandif}
d\eta(v,w)=v(\eta(w))-w(\eta(v))-\eta([v,w]).
\end{equation*}
In our case, the previous equation implies the following equalities
\begin{eqnarray}
d\theta_1(X,Y)&=&-\theta_1([X,Y])=0, \nonumber\\
d\theta_1(X,Z)&=&-\theta_1([X,Z])=0, \nonumber\\
d\theta_1(Y,Z)&=&-\theta_1([Y,Z])=0, \nonumber\\
d\theta_2(X,Y)&=&-\theta_2([X,Y])=0, \nonumber\\
d\theta_2(X,Z)&=&-\theta_2([X,Z])=0, \nonumber\\
d\theta_2(Y,Z)&=&-\theta_2([Y,Z])=0, \nonumber\\
d\theta_3(X,Y)&=&-\theta_3([X,Y])=\frac12(\varkappa-\widehat\varkappa), \nonumber\\
d\theta_3(X,Z)&=&-\theta_3([X,Z])=0, \nonumber\\
d\theta_3(Y,Z)&=&-\theta_3([Y,Z])=0, \nonumber
\end{eqnarray}
It follows from these equations, that $[X,Z]$, $[Y,Z]$ belong to $\mathcal D$ and
$[X,Y]\notin\mathcal D$ if and only if the difference of curvatures $\varkappa-\widehat\varkappa$ does not vanish identically. In fact, counting dimensions, we see that $\spn\{X,Y,Z,[X,Y]\}=\ker\theta_1\cap\ker\theta_2$. It is clear from the choice of $Z$ that $[X,Y]\notin\ker d\pi$ since if $[X,Y]=kZ$ for some $k\in\mathbb{R}$, then $d\theta_3(X,Y)=-k\theta_3(Z)=0$ which contradicts our assumption. This implies that $\spn\{d\pi(X),d\pi(Y),d\pi([X,Y])\}=D_1$ is a distribution of rank 3. Analogously we obtain
\begin{eqnarray}
d\theta_1([X,Y],X)&=&-\theta_1([[X,Y],X])=0, \nonumber\\
d\theta_1([X,Y],Y)&=&-\theta_1([[X,Y],Y])=\theta_3([X,Y]), \nonumber\\
d\theta_2([X,Y],X)&=&-\theta_2([[X,Y],X])=-\theta_3([X,Y]), \nonumber\\
d\theta_2([X,Y],Y)&=&0. \nonumber
\end{eqnarray}
By similar considerations, we can see that
$$\spn\{X,Y,Z,[X,Y],[[X,Y],X],[[X,Y],Y]\}=T(F\times\widehat F),$$
which implies that
$$\spn\{d\pi(X),d\pi(Y),d\pi([X,Y]),d\pi([[X,Y],X]),d\pi([[X,Y],Y])\}=D_2,$$
is a distribution of rank 5.

These calculations imply that $D$ is of Cartan type whenever $\varkappa-\widehat\varkappa$ does not vanish identically. Since the configuration space $Q$ is 5-dimensional, the distribution $D$ is bracket generating and thus, by the Chow-Rashevski{\u\i} theorem we can completely solve the connectivity problem. In the case when $\varkappa=\widehat\varkappa$, the distribution $D$ is integrable and therefore $Q$ is foliated by submanifolds of dimension 2.

It is mentioned in \cite{BH}, that their construction does not depend on imbedding into Euclidean space, however no attempts are made to compare this definition to the one for imbedded manifolds.

We present a simple example, illustrating the above mentioned approach.

\begin{example}\label{2Dexample}
Let us consider the problem of the two dimensional sphere $S^2$ rolling over the Euclidean plane $\real^2$. We can embed these surfaces in the three dimensional Euclidean space $\real^3$ via the parameterizations
$$S^2=\{(\cos\theta\cos\varphi,\sin\theta\cos\varphi,\sin\varphi):-\pi<\theta\leq\pi,-\frac\pi2<\varphi\leq \frac\pi2\},$$
$$\real^2=\{(x,y,0):x,y\in\real\}.$$

It follows from straightforward computations that, in this case, we have
\begin{eqnarray}
\alpha_1=\cos\varphi d\theta, & \alpha_2=d\varphi, & \alpha_{21}=\sin\varphi d\theta; \nonumber\\
\widehat\alpha_1=dx, & \widehat\alpha_2=dy, & \widehat\alpha_{21}=0. \nonumber
\end{eqnarray}
Thus, equations~\eqref{2Dcase} take the form
\begin{equation*}\label{2Dsphere}
\cos\varphi d\theta-dx=d\varphi-dy=\sin\varphi d\theta=0.
\end{equation*}
It is easy to see that
$$d\alpha_{21}=\cos\varphi d\theta\wedge d\varphi=\alpha_1\wedge\alpha_2,\quad\quad d\widehat\alpha_{21}=0,$$
from which it follows that $\varkappa=1$ and $\widehat\varkappa=0$. Since the difference of the Gaussian curvatures does not vanish identically, we obtain the well-known result that it is always possible to achieve any configuration from a given one by rolling the sphere over the plane without slipping or twisting.
\end{example}

\section{Intrinsic rolling} \label{introll}

\subsection{Reformulation of the rolling motion in terms of bundles}

Both formulations of rolling maps given in \cite{AS}  and \cite{BH} only use the configuration space as a manifold of isometries of tangent spaces of $M$ and $\widehat{M}$, without taking into account the imbedding into an ambient space. However, neither of these descriptions attempts to give any justifications for why the ambient space may be ignored, nor do they attempt to compare the intrinsic definition and the extrinsic definition given for imbedded manifolds in~\cite{Sharpe}. We would like to find a reformulation of Definition~\ref{imbeddef2} in such a way that the conditions (i')-(vi') are stated both in terms of intrinsic conditions given on $Q$ and some additional conditions given on another bundle, that carries the information on imbedding.

The conditions imposed over a rolling $(x,g)$ by Definitions~1 and~2 are nontrivial in normal directions for the imbedding of the manifolds with codimension $\nu$ greater than 1. So, it is natural to suppose that the total configuration space of the rolling dynamics will have a normal component which will takes care of the action of $g$ on the normal bundle. Therefore, we make the following analogue construction, as we did for $Q$, in order to construct a principal bundle over $M\times \widehat{M}$ of isometries of the normal tangent space. We start from a pair of imbeddings $\iota: M \rightarrow \mathbb{R}^{n + \nu}$ and $\widehat{\iota}: \widehat{M} \rightarrow \mathbb{R}^{n + \nu}$, given as initial data. Let $\Phi$ be the principal $\SO(\nu)-$bundle over $M$, such that the fiber over a point $x \in M$ consists of all positively oriented orthonormal frames $\{\epsilon_\lambda(x) \}_{\lambda =1}^\nu$ spanning~$T_xM^\perp$. Let $\widehat{\Phi}$ be the principal $\SO(\nu)-$bundle similarly defined on~$\widehat{M}$. Likewise we did in Section~\ref{conspace}, identifying $(F \times \widehat{F})/\SO(n)$ with
\begin{equation} \label{Qdef}
Q = \left\{ \left. q \in \Isom_0^+(T_xM, T_{\widehat{x}} \widehat{M})
 \right| x \in M, \widehat{x} \in \widehat{M}  \right\},
\end{equation}
we identify $(\Phi \times \widehat{\Phi})/\SO(\nu)$ with
\begin{equation} \label{Pdef}
P_{\iota,\widehat{\iota}} := \left\{ \left. p \in \Isom_0^+(T_xM^\bot, T_{\widehat{x}} \widehat{M}^\bot)
\right| x \in M, \widehat{x} \in \widehat{M}  \right\}.
\end{equation}
The space $P_{\iota,\widehat{\iota}}$ is a principal $\SO(\nu)-$bundle over $M \times \widehat{M}$. The dimension of $P_{\iota,\widehat{\iota}}$ is $2n + \frac{\nu (\nu -1)}{2}$. The left and right actions of $\SO(\nu)$ on the fibers are defined by the corresponding actions of $\SO(\nu)$ on $\widehat{\Phi}$ and $\Phi$, respectively, in a similar way to the left and right action on $Q$ described in Section~\ref{conspace}. We notice and reflect it in notations that $Q$ is invariant of imbeddings, while $P_{\iota,\widehat{\iota}}$ is not.

\begin{proposition} \label{Prop gtoqp}
If a curve $(x,g): [0,\tau] \rightarrow M \times \Isom^+(\mathbb{R}^{n+\nu})$ satisfies (i')-(vi') of Definition~\ref{imbeddef2},    then the mapping
$$t \mapsto (dg(t)|_{T_{x(t)}M},dg(t)|_{T_{x(t)}M^\bot}) = : \left(q(t),p(t) \right), $$
defines a curve in $Q \oplus P_{\iota,\widehat{\iota}}$ with the following properties:
\begin{itemize}
\item[(I)] no slip condition: $\dot{\widehat{x}}(t) = q(t) \dot{x}(t)$ for almost every $t$.
\item[(II)] no twist condition (tangential part):
$q(t) \frac{D}{dt} Z(t) = \frac{D}{dt} q(t) Z(t)$ for any tangent vector field
$Z(t)$ along $x(t)$ and almost every $t$.
\item[(III)] no twist condition (normal part):
$p(t) \frac{D^\bot}{dt} \Psi(t) = \frac{D^\bot}{dt} p(t) \Psi(t)$ for any normal vector field $\Psi(t)$
along $x(t)$ and almost every $t$.
\end{itemize}

Conversely, if $(q,p):[0,\tau] \to Q \oplus P_{\iota,\widehat{\iota}}$ is an absolutely continuous curve  satisfying (I)-(III), then there exists a unique rolling $(x,g):[0,\tau] \to M \times \Isom^+(\real^{n+\nu})$, such that $dg(t)|_{T_{x(t)} M} =q(t)$ and $dg(t)|_{T_{x(t)} M^\bot} =p(t)$.
\end{proposition}

\begin{proof}
Assume that $(x,g):[0,\tau] \to M \times \Isom^+(\real^{n+\nu})$ is a rolling map satisfying (i')-(vi').
The statements (i') and (ii') assure  that
\begin{equation}\label{orient}
\begin{split}
& dg(t)|_{T_{x(t)} M}  \in \Isom(T_{x(t)} M,T_{\widehat{x}(t)} \widehat{M}) \quad \text{and}\\
& dg(t)|_{T_{x(t)} M^\bot}  \in \Isom(T_{x(t)} M^\bot,T_{\widehat{x}(t)} \widehat{M}^\bot).
\end{split}
\end{equation}
Since $dg(t)$ must be orientation preserving in $\mathbb R^{n+\nu}$ we conclude that both of the mappings~\eqref{orient} are either orientation reversing or orientation preserving. The additional requirement (vi') implies that $(q,p)$ is orientation preserving. The conditions (I)-(III) correspond to the conditions (iii')-(v').

Conversely, if we have a curve $(q(t),p(t))$ in $Q \oplus P_{\iota,\widehat{\iota}}$ with projection $(x(t),\widehat{x}(t))$ into the product manifold $M \times \widehat{M}$, then we may construct the isomorphism $g \in \Isom^+(\real^{n+\nu})$ in the following way. We write $g(t): \bar{x} \mapsto \bar{A}(t) \bar{x} + \bar{r}(t), \bar{A}(t) \in \SO(n+\nu)$, where $\bar A(t)=dg(t)$ is determined by the conditions
$$dg(t)|_{T_{x(t)} M} =q(t)|_{T_{x(t)} M}, \qquad  dg(t)|_{T_{x(t)} M^\bot} = p(t)|_{T_{x(t)} M^\bot}.$$ Then
$$\im dg(t)|_{T_{x(t)} M} =T_{\widehat{x}(t)} \widehat{M}, \qquad \im dg(t)|_{T_{x(t)} M^\bot} = T_{\widehat{x}(t)} \widehat{M}^\bot.$$
The vector $\bar{r}(t)$ is determined by $\bar{r}(t) = \widehat{x}(t) - A(t) x(t)$.
\end{proof}

The one-to-one correspondence between rolling maps and absolutely continuous curves in $Q \oplus P_{\iota,\widehat{\iota}}$, satisfying (I)-(III), naturally leads to a definition of a rolling map in terms of these bundles.

\begin{definition} \label{imbeddef3}
A rolling of $M$ on $\widehat{M}$ without slipping or twisting is an absolutely continuous curve $(q,p):[0,\tau] \to Q \oplus
P_{\iota,\widehat{\iota}}$ such that $(q(t),p(t))$ satisfies
\begin{itemize}
\item[(I)] no slip condition: $\dot{\widehat{x}}(t) = q(t) \dot{x}(t)$ for almost every $t$,
\item[(II)] no twist condition (tangential part):
$q(t) \frac{D}{dt} Z(t) = \frac{D}{dt} q(t) Z(t)$
for any tangent vector field $Z(t)$ along $x(t)$ and almost every $t$,
\item[(III)] no twist condition (normal part): $p(t) \frac{D^\bot}{dt} \Psi(t) =
\frac{D^\bot}{dt} p(t) \Psi(t)$ for any normal vector field $\Psi(t)$ along $x(t)$ and almost every $t$.
\end{itemize}
\end{definition}

A purely intrinsic definition of a rolling is deduced from Definition~\ref{imbeddef3}, by restricting it to the bundle~$Q$. This concept naturally generalizes the definition given in \cite{AS} for 2-dimensional Riemannian manifolds imbedded into~$\mathbb{R}^3$ and we use the term {\it intrinsic rolling} for this object.

\begin{definition} \label{intrinsdef}
An intrinsic rolling of two $n$-dimensional oriented Riemannian manifolds $M$ on $\widehat{M}$ without slipping or twisting is an absolutely continuous curve $q:[0,\tau] \rightarrow Q$, satisfying the following conditions: if $x(t) = \pr_M q(t)$ and \ $\widehat{x}(t) = \pr_{\widehat{M}} q(t)$, then
\item[(I')] no slip condition: $\dot{\widehat{x}}(t) = q(t) \dot{x}(t)$ for almost all $t$,
\item[(II')] no twist condition: $Z(t)$ is a parallel tangent vector field along $x(t)$, if and only if $q(t) Z(t)$ is parallel along $\widehat{x}(t)$ for almost all $t$.
\end{definition}

\subsection{Rolling versus intrinsic rolling along the same curves}
Suppose that the projection of a rolling map into $M \times \widehat{M}$ is a fixed pair of curves. Questions that naturally arise are:
\begin{itemize}
\item If $(q_1(t), p_1(t))$ and $(q_2(t), p_2(t))$ are two rollings of $M$ on $\widehat{M}$, along $x(t)$ and
$\widehat{x}(t)$, how do they relate to one another? How many of the properties of the rolling are fixed by choosing paths?
\item Suppose that an intrinsic rolling $q(t)$ and two imbeddings, $\iota:M \to \real^{n+\nu}$ and $\widehat{\iota}:
\widehat{M} \to \real^{n+\nu}$, are given. When can the intrinsic rolling $q(t)$ be extended to a rolling $(q(t),p(t))$? Is this extension unique?
\end{itemize}
Before we start working with this, let us consider the following simple example, where the different imbeddings are easy to picture.

\begin{example} \label{S1onR}
Let us consider $\widehat{M} = \real$, with the usual Euclidean structure, and $M = S^1$, with the subspace metric, when considered as the unit circle in $\real^2$, with positive orientation counterclockwise. Let $x:[0,\tau] \to S^1$ be written as $x(t) = e^{i\varphi(t)}$, $\varphi:[0,\tau] \to \mathbb{R}$ being an absolutely continuous function. Since $\SO(1)$ is just the trivial group, $Q \cong M \times \widehat{M}$. It is clear from the no-slipping condition that
$$\widehat{x}(t) = \widehat{x}(0) + \varphi(t) - \varphi(0).$$
Without loss of of generality, we may assume $\widehat{x}(0) = \varphi(0) = 0$. We consider the possible rollings under different imbeddings. In the following cases, $e_1$ and $\hat{e}_1$ will always be positively oriented unit basis vectors for $TM$ and $T\widehat{M}$ respectively (when they are seen as sub-bundles of $T\real^{1+\nu}$ restricted to either $M$ or $\widehat{M}$), while $\{ \epsilon_\lambda \}_{\lambda = 1}^\nu$ and $\{ \hat{\epsilon}_\kappa \}_{\kappa = 1}^\nu$  are positively oriented bases of $TM^\perp$ and $T\widehat{M}^\perp$.
The coordinates of $\real^{1+\nu}$ will be denoted by $(\bar{x}_1, \dots, \bar{x}_n)$.
\begin{itemize}
\item[\it Case 1:] Let us consider the simplest example, with
$$\iota_1:M \to \real^2, \qquad \iota_1: e^{i\varphi} \mapsto (\sin \varphi, 1 - \cos \varphi),$$
$$\widehat{\iota}_1: \widehat{M} \to \real^2, \qquad \widehat{\iota}_1: \widehat{x} \mapsto (\widehat{x}, 0).$$
Then
$$e_1(e^{i\varphi}) = \cos \varphi \frac{\partial}{\partial \bar{x}_1}(\iota_1(e^{i\varphi})) +
\sin \varphi \frac{\partial}{\partial \bar{x}_2}(\iota_1(e^{i\varphi})),$$
$$\epsilon_1(e^{i\varphi}) = - \sin \varphi \frac{\partial}{\partial \bar{x}_1}(\iota_1(e^{i\varphi})) +
\cos \varphi \frac{\partial}{\partial \bar{x}_2}(\iota_1(e^{i\varphi})),$$
$$\hat{e}_1(\widehat{x}) = \frac{\partial}{\partial \bar{x}_1}(\widehat{\iota}_1(\widehat{x})) ,
\qquad \hat{\epsilon}_1(\widehat{x}) = \frac{\partial}{\partial \bar{x}_2}(\widehat{\iota}_1(\widehat{x})).$$
Here, also $\SO(\nu)$ is trivial, so there is so there is only one way to roll.

\item[\it Case 2:] We  do the same imbeddings as above, only increasing the codimension by one.
$$\iota_2:M \to \real^3, \qquad \iota_2: e^{i\varphi} \mapsto (\sin \varphi, 1 - \cos \varphi, 0),$$
$$\widehat{\iota}_2: \widehat{M} \to \real^3, \qquad \widehat{\iota}_2: \widehat{x} \mapsto (\widehat{x}, 0, 0).$$
Then
$$e_1(e^{i\varphi}) = \cos \varphi \frac{\partial}{\partial \bar{x}_1}(\iota_2(e^{i\varphi})) +
\sin \varphi \frac{\partial}{\partial \bar{x}_2}(\iota_2(e^{i\varphi}),$$
$$\epsilon_1(e^{i\varphi}) = - \sin \varphi \frac{\partial}{\partial \bar{x}_1}(\iota_2(e^{i\varphi})) +
\cos \varphi \frac{\partial}{\partial \bar{x}_2}(\iota_2(e^{i\varphi})), $$
$$\epsilon_2(e^{i\varphi}) = \frac{\partial}{\partial \bar{x}_3}(\iota_2(e^{i\varphi})),$$
$$\hat{e}_1(\widehat{x}) = \frac{\partial}{\partial \bar{x}_1}(\widehat{\iota}_2(\widehat{x})),
\qquad \hat{\epsilon}_1(\widehat{x}) = \frac{\partial}{\partial \bar{x}_2}(\widehat{\iota}_2(\widehat{x})),
\qquad \hat{\epsilon}_2(\widehat{x}) = \frac{\partial}{\partial \bar{x}_3}(\widehat{\iota}_2(\widehat{x})).$$
Now we know that the matrix representation $B$ of $p(t)$ with respect to the bases $\{ e_\lambda \}_{\lambda = 1}^\nu$ and $\{ \hat{e}_\kappa \}_{\kappa = 1}^\nu$, can be represented as
$$B = \left(\begin{array}{cc} \left\langle \hat{e}_1 , p(t) e_1 \right\rangle & \left\langle\hat{e}_1 , p(t) e_2 \right\rangle \\
\left\langle \hat{e}_2 , p(t) e_1 \right\rangle & \left\langle \hat{e}_2 , p(t) e_2 \right\rangle\end{array} \right)
= \left(\begin{array}{cc} \cos \theta(t) & \sin \theta(t) \\
- \sin \theta(t) & \cos \theta(t) \end{array} \right) \in \SO(2).$$
We calculate the restrictions of $\theta(t)$ given by (III).
$$p(t) \frac{D}{dt} \epsilon_1(x(t)) = p(t) \nabla^\perp_{\dot{x}(t)} \epsilon_1 = 0 = \frac{D^\perp}{dt} p(t) \epsilon_1(x(t))$$
$$= - \dot{\theta}(t) (\sin \theta(t) \hat{\epsilon}_1 + \cos \theta(t) \hat{\epsilon}_1)
+ \cos \theta(t) \nabla^{\perp}_{\dot{\widehat{x}}(t)} \hat{\epsilon}_1 - \sin \theta(t) \nabla^{\perp}_{\dot{\widehat{x}}(t)} \hat{\epsilon}_2$$
$$= - \dot{\theta}(t) (\sin \theta(t) \hat{\epsilon}_1 + \cos \theta(t) \hat{\epsilon}_2),$$
for almost every $t$, so $\theta(t)$ is a constant.

\item[\it Case 3:] We continue with $\nu = 2$, but change the imbedding of $\widehat{M}$
to a spiral.
$$\iota_2:M \to \real^3, \qquad \iota_2: e^{i\varphi} \mapsto (\sin \varphi, 1 - \cos \varphi, 0),$$
$$\widehat{\iota}_3: \widehat{M} \to \real^3, \qquad \iota_3: \widehat{x} \mapsto
\frac{1}{\sqrt{2}} (\cos \widehat{x}, \sin \widehat{x}, \widehat{x}).$$
Then
$$e_1(e^{i\varphi}) = \cos \varphi \frac{\partial}{\partial \bar{x}_1}(\iota_2(e^{i\varphi})) +
\sin \varphi \frac{\partial}{\partial \bar{x}_2}(\iota_2(e^{i\varphi}),$$
$$\epsilon_1(e^{i\varphi}) = - \sin \varphi \frac{\partial}{\partial \bar{x}_1}(\iota_2(e^{i\varphi})) +
\cos \varphi \frac{\partial}{\partial \bar{x}_2}(\iota_2(e^{i\varphi})), $$
$$\epsilon_2(e^{i\varphi}) = \frac{\partial}{\partial \bar{x}_3}(\iota_2(e^{i\varphi})),$$
$$\hat{e}_1(\widehat{x}) = \frac{1}{\sqrt{2}} \left(- \sin \widehat{x} \frac{\partial}{\partial \bar{x}_1}(\widehat{\iota}_2(\widehat{x})) + \cos \widehat{x} \frac{\partial}{\partial \bar{x}_2}(\widehat{\iota}_3(\widehat{x}))
+ \frac{\partial}{\partial \bar{x}_3}(\widehat{\iota}_3(\widehat{x}))
\right),$$
$$\hat{\epsilon}_1(\widehat{x}) = \frac{1}{\sqrt{2}} \left(- \sin \widehat{x} \frac{\partial}{\partial \bar{x}_1}(\widehat{\iota}_2(\widehat{x})) + \cos \widehat{x} \frac{\partial}{\partial \bar{x}_2}(\widehat{\iota}_3(\widehat{x}))
- \frac{\partial}{\partial \bar{x}_3}(\widehat{\iota}_3(\widehat{x}))
\right),$$
$$\hat{\epsilon}_2(\widehat{x}) = -\cos \widehat{x} \frac{\partial}{\partial \bar{x}_1}(\widehat{\iota}_2(\widehat{x})) - \sin \widehat{x} \frac{\partial}{\partial \bar{x}_2}(\widehat{\iota}_3(\widehat{x})).$$
We have the same matrix representation of $p(t)$,
$$B = \left(\begin{array}{cc} \left\langle \hat{e}_1 , p(t) e_1 \right\rangle & \left\langle\hat{e}_1 , p(t) e_2 \right\rangle \\
\left\langle \hat{e}_2 , p(t) e_1 \right\rangle & \left\langle \hat{e}_2 , p(t) e_2 \right\rangle\end{array} \right)
= \left(\begin{array}{cc} \cos \theta(t) & \sin \theta(t) \\
- \sin \theta(t) & \cos \theta(t) \end{array} \right) \in \SO(2).$$
We calculate the restrictions of $\theta(t)$ given by (III).
$$p(t) \nabla^\perp_{\dot{x}(t)} \epsilon_1 = 0 = \frac{D^\perp}{dt} p(t) \epsilon_1$$
$$= - \dot{\theta}(t) (\sin \theta(t) \hat{\epsilon}_1 + \cos \theta(t) \hat{\epsilon}_1)
+ \cos \theta(t) \nabla^{\perp}_{\dot{\widehat{x}}(t)} \hat{\epsilon}_1 - \sin \theta(t) \nabla^{\perp}_{\dot{\widehat{x}}(t)} \hat{\epsilon}_2$$
$$= \left(\frac{\dot{\widehat{x}}(t)}{\sqrt{2}}- \dot{\theta}(t) \right)
 (\sin \theta(t) \hat{\epsilon}_1 + \cos \theta(t) \hat{\epsilon}_2),$$
so $\theta(t) = \theta_0 + \frac{1}{\sqrt{2}}\widehat{x}(t)$. So now, the circle $M$ will rotate along the spiral $\widehat{M}$,
but its path is determined by the initial angle. Notice also that if we define a new orthonormal frame of
$T\widehat{M}^\perp$ by
$$\widehat{\Upsilon}_1 = \cos \left(\frac{\widehat{x}}{\sqrt{2}} \right) \hat{\epsilon}_1 - \sin
\left(\frac{\widehat{x}}{\sqrt{2}} \right) \hat{\epsilon}_2,$$
$$\widehat{\Upsilon}_2 = \sin \left(\frac{\widehat{x}}{\sqrt{2}} \right) \hat{\epsilon}_1
+ \cos \left(\frac{\widehat{x}}{\sqrt{2}} \right) \hat{\epsilon}_2,$$
then $p(t)$ becomes a constant matrix with respect to the bases $\epsilon_1, \epsilon_2$ and
$\widehat{\Upsilon}_1, \widehat{\Upsilon}_2$.
\end{itemize}
We see that for cases above, the intrinsic rolling $t \mapsto (e^{i\varphi(t)}, \varphi(t))$ either uniquely induces a rolling, or the rolling is determined by an initial configuration of the normal tangent spaces given by $\theta(0) = \theta_0$. Note also that we are able to find a choice of bases so that $p(t)$ is constant with respect to this basis. Notice that these bases consist of normal parallel vector fields.
\end{example}


We continue to work with oriented manifolds $M$ and $\widehat{M}$ imbedded in $\mathbb{R}^{n+\nu}$ and containing curves $x(t)$ and $\widehat{x}(t)$, respectively. In the remaining of this section we will use the following notations: $\{e_j(t)\}_{j =1}^n$ will be a collection of parallel tangent vector fields along $x(t)$ that  forms an orthonormal basis for $T_{x(t)}M$ at each point of $M$, $\{\epsilon_\lambda(t)\}_{\lambda =1}^\nu$ will be a collection of normal parallel vector fields along $x(t)$ forming an orthonormal basis for $T_{x(t)}M^\bot$. We know that we can construct such vector fields by parallel transport and normal parallel transport along $x(t)$. Parallel frames $\{\hat{e}_i\}_{i =1}^n$ and $\{\hat{\epsilon}_\kappa\}_{\kappa =1}^\nu$ will be defined similarly along $\widehat{x}(t)$. Recall that Latin indices $i,j,\dots$ always go from $1$ to $n$, while Greek ones $\kappa,\lambda,\dots$ vary from 1 to $\nu$.

The following lemma reflects that a rolling map preserves parallel vector fields. Namely, the image of a parallel frame over $M$ has constant coordinates in a parallel frame over $\widehat M$.
\begin{lemma} \label{constantM}
A curve $(q(t),p(t))$ in $Q \oplus P_{\iota,\widehat{\iota}}$ in the fibers over $(x(t),\widehat{x}(t))$, satisfies (II) and (III) if and only if the matrices $A(t) = (a_{ij}(t))$ and $B(t) = (b_{\kappa\lambda}(t))$, defined by
$$a_{ij}(t) = \hat{e}_i^*(t) q(t) e_j(t), \qquad b_{\kappa \lambda}(t) = \hat{\epsilon}_\kappa^*(t) p(t) \epsilon_\lambda(t),$$
are constant.
\end{lemma}

\begin{proof}
Let $(q(t),p(t))$ be an absolutely continuous curve. Then we have $\langle \hat{e}_i , \dot{\hat{e}}_j \rangle = \langle e_i , \dot{e}_j \rangle = 0$ and
$$\dot{a}_{ij}(t) = \langle \dot{\hat{e}}_i, q(t) e_j \rangle + \left\langle \hat{e}_i, \frac{d}{dt} (q(t) e_j) \right\rangle$$
by the product rule. The vectors $q(t)^{-1}\hat e_i$, $q(t)e_j$ are tangent, so $\langle q(t)^{-1}\hat e_i , \dot{e}_j \rangle=\langle \dot{\hat e}_i,q(t) e_j \rangle=0$ and
\begin{eqnarray*}\dot{a}_{ij}(t)&  =  & \langle \hat{e}_i, \dot q(t) e_j \rangle+\langle \hat{e}_i, q(t) \dot e_j \rangle + \langle \dot{\hat{e}}_i, q(t) e_j \rangle \\
&
= &\langle \hat{e}_i, \dot q(t) e_j \rangle+\langle q(t)^{-1}\hat e_i , \dot{e}_j \rangle =
\left\langle \hat{e}_i, \frac{d}{dt} \big(q(t) e_j\big) - q(t) \dot{e}_j \right\rangle \\
& = &\left\langle \hat{e}_i, \frac{D}{dt} q(t) e_j - q(t) \frac{D}{dt} e_j \right\rangle=0.
\end{eqnarray*}
So (II) holds if and only if $\dot{a}_{ij}(t) = 0.$ Similar result holds for the basis of the normal tangent bundle.
\end{proof}

The following two theorems give the answer to the questions raised at the beginning of this section.
\begin{theorem} \label{number intrinsic}
Let $q_0:[0,\tau] \rightarrow Q$ be a given intrinsic rolling map without slipping or twisting with the projection $\pr_{M \times \widehat{M}} q_0(t) = (x(t),\widehat{x}(t))$. Denote by $k$ the dimension of the space of parallel tangent vector fields along $\widehat{x}(t)$, that are orthogonal to $\dot{\widehat{x}}(t)$ whenever they are defined. Then the following statements hold.
\begin{itemize}
\item[(a)] The map $q_0$ is the unique intrinsic rolling of $M$ on $\widehat{M}$ along $x(t)$ and $\widehat{x}(t)$ if and only if $k \leq 1$.
\item[(b)] If $k \geq 2$, then there exists an injective Lie group homomorphism $\zeta: \SO(k) \rightarrow \SO(n)$ such that for each $A' \in \SO(k)$ the map $q_0(t) \zeta(A')$ is an intrinsic rolling over $(x(t),\widehat{x}(t))$, and any intrinsic rolling over $(x(t),\widehat{x}(t))$ is of this form.
\end{itemize}
\end{theorem}

\begin{proof}

Pick up frames of parallel vector fields $\{e_i\}_{i=1}^n$ and $\{\hat{e}_i \}_{i=1}^n$ along $x(t)$ and $\widehat{x}(t)$, respectively, such that $q_0(t) e_i = \hat{e}_i$. This is possible due to Lemma~\ref{constantM}. We also choose the frames in a way that the $k$ first vector fields are orthogonal to $\dot{\widehat{x}}$.

Writing $\dot{\widehat{x}} = \sum_{i =1}^n \dot{\widehat{x}}_i(t) \hat{e}_i(t)$ and $\dot{x} = \sum_{i =1}^n \dot{x}_i(t) e_i(t)$, we get $\dot{\widehat{x}}_i(t) = \dot{x}_i(t)$ and $\dot{\widehat{x}}_1(t) = \dots = \dot{\widehat{x}}_k(t) =0$. So, if $q$ is any other rolling, then $A =(a_{ij})= (\langle\hat{e}_i(t) , q(t) \, e_j(t)  \rangle)$ is clearly of the form
$$A = \left( \begin{array}{cc} A' & 0 \\ 0 & \mathbf 1_{n-k} \end{array} \right), \qquad A' \in \SO(k),$$
where $\mathbf 1_{n-k}$ is the $\big((n-k)\times(n-k)\big)$-unit matrix.

The converse also holds; that is, for any such matrix $A$, there is a rolling corresponding to it.
\end{proof}

\begin{theorem} \label{unique extension}
 Let $q_0:[0,\tau] \rightarrow Q$ be an intrinsic rolling and let $\iota:M \rightarrow \mathbb{R}^{n+\nu}$ and $\widehat{\iota}: \widehat{M} \rightarrow \mathbb{R}^{n+\nu}$ be given imbeddings. Then there exists a rolling $(q_0,p): [0,\tau] \rightarrow Q \oplus P_{\iota,\widehat{\iota}}$ that is unique up to a right action of $\SO(\nu)$.
\end{theorem}

\begin{proof}
We pick up normal parallel frames $\{\epsilon_\lambda(t)\}_{\lambda =1}^\nu$ and $\{\hat{\epsilon}_\kappa(t)\}_{\kappa =1}^\nu$ along $x(t)$ and $\widehat{x}(t)$, respectively. For any element $B \in \SO(\nu)$ we define $p(t)$ by
$$B = \langle\widehat{\epsilon}_\kappa , p(t) \, \epsilon_\lambda \rangle.$$ The map $(q_0(t),p(t))$ is a rolling by Lemma~\ref{constantM}, and all rollings are of this form.
\end{proof}

\begin{corollary}
Assume that $x(t)$ is a geodesic in $M$. Then there exists an intrinsic rolling of $M$ on $\widehat{M}$ along $(x(t),\widehat{x}(t))$ if and only if $\widehat{x}(t)$ is a geodesic with the same speed as $x(t)$. Moreover, if $n \geq 2$ then there is an injective Lie group homomorphism $\zeta\colon\SO(n-1) \to \SO(n)$, such that all intrinsic rollings over $(x(t),\widehat{x}(t))$ differ by an element in $\im \zeta$.
\end{corollary}

\begin{proof}
Taking into account the equality $\frac{D}{dt} \dot{\widehat{x}}(t) = \frac{D}{dt} q(t) \dot{x}(t)=
q(t) \frac{D}{dt} \dot{x}(t)$, we conclude that if $x(t)$ is a geodesic then $\widehat{x}(t)$ is also geodesic. In order to satisfy (I) we need to require that the speed of $\dot{\widehat{x}}(t)$ is the same as the speed of $\dot x(t)$. Conversely, the equality of speeds implies condition~(I).

We start the construction of rolling map by choosing $e_1(t) = \frac{\dot x(t)}{\langle \dot{x}(t), \dot{x}(t) \rangle}$ that is parallel along $x(t)$. The remaining $n-1$ parallel vector fields we pick up in a way that they form an orthonormal basis together with $e_1(t)$ along the curve $x(t)$. We repeat the same construction for a parallel frame $\{ \hat{e}_i(t) \}_{i=1}^n$ along $\widehat{x}(t)$. Define the intrinsic rolling $q(t)$ by
\begin{equation} \label{intr1 eq}
\begin{split}
& \hat{e}_1^*(t) \, q(t) \, e_j(t) = \hat{e}_j^*(t) \, q(t) \, e_1(t) = \delta_{1,j},\\
& \ \ \ \ \ \ A' = \left(\hat{e}_{i+1}^*(t) \, q(t) \, e_{j+1}(t) \right)_{i,j =1}^{n-1},
\end{split}
\end{equation}
where $A' \in \SO(n-1)$ will be a constant matrix. Conversely, we can construct a rolling
by formulas~\eqref{intr1 eq} starting from $A' \in \SO(n-1)$.
\end{proof}

\section{Distributions for rolling and intrinsic rolling maps}

The aim of this Section is to formulate the kinematic conditions of no-slipping and no-twisting in terms of a distribution. In this setting, a rolling will be an absolutely continuous curve almost everywhere tangent to this distribution.

\subsection{Local trivializations of $Q$} \label{local triv}

Let $\pi: Q \oplus P_{\iota,\widehat{\iota}} \to M \times \widehat{M}$ denote the canonical projection.
Consider a rolling $\gamma(t) = (q(t) , p(t))$, then $\pi\circ \gamma(t) = \left( x(t) , \widehat{x}(t) \right)$.
Given an arbitrary $t_0$ in the domain of $\gamma(t)$, let $U$ and $\widehat{U}$ denote neighborhoods of $x(t_0)$ and $\widehat{x}(t_0)$
in $M$ and $\widehat{M}$, respectively,
such that the both bundles $TM \to M$
and $TM^\perp \to M$ trivialize being restricted to $U$. In the same way we chose $\widehat U$, such that
both $T\widehat{M} \to \widehat{M}$ and $T\widehat{M}^\perp \to \widehat{M}$ trivialize when they are restricted to
$\widehat{U}$. This implies that the bundle $Q \oplus P_{\iota,\widehat{\iota}} \to M \times \widehat{M}$, trivializes
when it is restricted to $U \times \widehat{U}$.
To see this, let $\{ e_j \}_{j = 1}^n$, $\{ \epsilon_\lambda \}_{\lambda = 1}^\nu$, $\{ \hat{e}_i \}_{i = 1}^n$ and
$\{ \hat{\epsilon}_\kappa \}_{\kappa = 1}^\nu$ denote positively oriented
orthonormal bases of vector fields of $TM|_U$,
$TM^\perp|_U$, $T\widehat{M}|_{\widehat{U}}$ and $T\widehat{M}^\perp|_{\widehat{U}}$, respectively.
Then there is a trivialization
\begin{equation} \label{trivi}
\begin{array}{rcl}
Q \oplus P_{\iota,\widehat{\iota}}|_ {U \times \widehat{U}} & \stackrel{h}{\to} & U \times \widehat{U} \times \SO(n) \times \SO(\nu)
\\
(q, p) & \mapsto & (x,\widehat{x}, A, B),
\end{array}
\end{equation}
given by projections
$$x = \pr_U (q,p), \qquad \widehat{x} = \pr_{\widehat{U}} (q,p),$$
$$A = \left( a_{ij} \right)_{i,j = 1}^n = \left( \langle q e_j, \hat{e}_i \rangle \right)_{i,j = 1}^n, \qquad
B = \left( b_{\kappa\lambda} \right)_{\kappa,\lambda = 1}^\nu =
\left( \langle p \epsilon_\lambda, \hat{\epsilon}_\kappa \rangle \right)_{\kappa,\lambda = 1}^\nu.$$
The domain of $\gamma$ can be chosen connected, containing $t_0$, and such that its image lies in
$\pi^{-1}(U \times \widehat{U})$. Let us identify $\gamma(t)$ with its image under the trivialization
given by $(x(t), \widehat{x}(t), A(t), B(t))$.

Each of the requirements (I)-(III) can be written as restrictions to $\dot{\gamma}(t)$.
We will show, that all admissible values of $\dot\gamma(t)$ form a distribution; that is a smooth
sub-bundle, of $T(Q \oplus P_{\iota,\widehat{\iota}})$. We will use the local trivializations to
describe this distribution.

\subsection{The tangent space of $\SO(n)$} \label{tangent SOn}
Let $U$ and $\widehat{U}$ be as in Section~\ref{local triv}. Then we get  in trivialization
$$T\pi^{-1}(U \times \widehat{U}) = TU \times T\widehat{U} \times T\SO(n) \times T\SO(\nu).$$
The decomposition requires that we present a detailed description of the tangent space of $\SO(n)$ in terms of left and right invariant vector fields.

We start by considering the imbedding of $\SO(n)$ in $\GL(n)$, the group of invertible real $n \times n$ matrices.
Denote the matrix entries of a matrix $A$ by $ (a_{ij})$ and the transpose matrix by $A^t$. Then, differentiating the condition $A^tA = {\bf 1}$, we obtain
$$T\SO(n) = \bigcap_{i \leq j} \ker \omega_{ij}, \qquad \omega_{ij} =
\sum_{r =1}^n \left(a_{rj} \, da_{ri} + a_{ri} \, da_{rj}\right).$$
It is clear that the tangent space at the identity $1$ of $\SO(n)$ is spanned by
$$W_{ij}(1) := \frac{\partial}{\partial a_{ij}} - \frac{\partial}{\partial a_{ji}}, \qquad 1 \leq i < j \leq n.$$
We denote $\mathfrak{so}(n)=\spn\{W_{ij}(1)\}$ following the classical notation. We use the left translation of these vector to define
\begin{equation} \label{SObasis}
W_{ij}(A) := A \cdot W_{ij}(1) = \sum_{r =1}^n \left(a_{ri} \frac{\partial}{\partial a_{rj}} - a_{rj} \frac{\partial}{\partial a_{ri}} \right)
\end{equation}
as global left invariant basis of $T\SO(n)$.
Note that the left and right action in $T\SO(n)$ is described by
$$A \cdot \frac{\partial}{\partial a_{ij}} = \sum_{r=1}^n a_{ri} \frac{\partial}{\partial a_{rj}} \qquad \qquad
\frac{\partial}{\partial a_{ij}} \cdot A = \sum_{s =1}^n a_{js} \frac{\partial}{\partial a_{is}}.$$
We have the following formula to switch from left to right translation
$$A \cdot \frac{\partial}{\partial a_{ij}} = \sum_{r=1}^n a_{ri} \frac{\partial}{\partial a_{rj}}
= \sum_{l, r =1}^n a_{ri} \delta{j,l} \frac{\partial}{\partial a_{rl}}
= \sum_{l, r, s =1}^n a_{ri} a_{si} a_{sl} \frac{\partial}{\partial a_{rl}}$$
$$= \sum_{r,s=1}^n a_{ri} a_{si} \left(\frac{\partial}{\partial a_{rs}} \cdot A \right),$$
and the other way around,
$$\frac{\partial}{\partial a_{ij}} \cdot A = \sum_{s =1}^n a_{js} \frac{\partial}{\partial a_{is}}
= \sum_{l, s =1}^n a_{js} \delta_{i,l} \frac{\partial}{\partial a_{ls}}
= \sum_{l, r, s =1}^n a_{js} a_{ir} a_{lr} \frac{\partial}{\partial a_{ls}}$$
$$= \sum_{r, s =1}^n a_{js} a_{ir} \left(A \cdot \frac{\partial}{\partial a_{rs}}\right).$$

Therefore, the right invariant basis of $T\SO(n)$ can be written as
$$W_{ij}(1) \cdot A = \mathrm{Ad}(A^{-1}) W_{ij} (A) = \sum_{r<s} (a_{ir} a_{js} - a_{jr} a_{is}) W_{rs}(A).$$

If we let $W_{ij}$ be defined \eqref{SObasis} also when $i$ is not less then $j$, (so $W_{ij} = - W_{ji}$) then
the bracket relations are given by
$$[W_{ij}, W_{kl}] = \delta_{j,k} W_{il} + \delta_{i,l} W_{jk} - \delta_{i,k} W_{jl} - \delta_{j,l} W_{ik}.$$

\subsection{Distributions} Now we are ready to rewrite the kinematic conditions (I)-(III) as a distribution. Let $\gamma(t)$ be a rolling satisfying the conditions (I)-(III). Consider it image under the trivializations. Then
\begin{equation}\label{dotgamma}
\dot{\gamma}(t) = \dot{x}(t) + \dot{\widehat{x}}(t) + \sum_{i,j = 1}^n \dot{a}_{ij} \frac{\partial}{\partial a_{ij}}
+ \sum_{\kappa,\lambda = 1}^\nu \dot{b}_{\kappa\lambda} \frac{\partial}{\partial b_{\kappa\lambda}}.
\end{equation}
If we denote $\dot{x}(t)$ by $Z(t)$, then (I) holds if and only if $\dot{\widehat{x}}(t) = q(t) Z(t)$.

We want, basing on conditions (II) and (III), write the last two terms in~\eqref{dotgamma} in right invariant basis of corresponding tangent spaces of $\SO(n)$ and $\SO(\nu)$.
We start from (II) and remark that
$$q(t) e_j = \sum_{i =1}^n a_{ij}(t) \hat{e}_i, \quad \text{and} \quad q^{-1}(t) \hat{e}_i = \sum_{j=1}^n a_{ij}(t) e_j$$
for orthonormal bases $\{e_j\}_{j=1}^{n}$ and $\{\widehat e_j\}_{j=1}^{n}$. Condition (II) holds if and only if $q \frac{D}{dt} e_j(x(t)) = \frac{D}{dt} qe_j(x(t))$ for $j =1, \dots, n$,
that yields
$$0 = \left\langle q \frac{D}{dt} e_j(x(t))- \frac{D}{dt} q e_j(x(t)), \hat{e}_i \right\rangle$$
$$= \left\langle \nabla_{Z(t)} e_j, q^{-1} \hat{e}_i \right\rangle
- \left\langle \sum_{l=1}^n \dot{a}_{lj} \hat{e}_l , \hat{e}_i \right\rangle
- \left\langle \sum_{l=1}^n a_{lj} \nabla_{qZ(t)} \hat{e}_l, \hat{e}_i \right\rangle.$$
$$= \sum_{l = 1}^n a_{il} \left\langle \nabla_{Z(t)} e_j, e_l \right\rangle
- \dot{a}_{ij} - \sum_{l=1}^n a_{lj} \left\langle\nabla_{qZ(t)} \hat{e}_l, \hat{e}_i \right\rangle$$  for every $i,j =1, \dots, n$.
Hence, the third term in~\eqref{dotgamma} can be written as follows
\begin{eqnarray}\label{sumdist}
\sum_{i,j=1}^n \dot{a}_{ij} \frac{\partial}{\partial a_{ij}}&
=& \sum_{i,j=1}^n \left( \sum_{l = 1}^n a_{il} \left\langle \nabla_{Z(t)} e_j, e_l \right\rangle
 - \sum_{l=1}^n a_{lj} \left\langle\nabla_{qZ(t)} \hat{e}_l, \hat{e}_i \right\rangle\right) \frac{\partial}{\partial a_{ij}}\nonumber\\
&=& \sum_{j,l=1}^n \left\langle \nabla_{Z(t)} e_j, e_l \right\rangle A \cdot \frac{\partial}{\partial a_{lj}}
 - \sum_{i,l=1}^n \left\langle\nabla_{qZ(t)} \hat{e}_l, \hat{e}_i \right\rangle \frac{\partial}{\partial a_{il}} \cdot A\nonumber\\
&=& \sum_{i,j=1}^n \left\langle \nabla_{Z(t)} e_j, e_i \right\rangle A \cdot \frac{\partial}{\partial a_{ij}}
 - \sum_{i,j,r,s=1}^n a_{ir} a_{js} \left\langle\nabla_{qZ(t)} \hat{e}_j, \hat{e}_i \right\rangle A \cdot \frac{\partial}{\partial a_{rs}}\nonumber\\
&=& \sum_{i,j=1}^n \left(\left\langle \nabla_{Z(t)} e_j, e_i \right\rangle
 - \sum_{s = 1}^n a_{sj}  \left\langle \nabla_{qZ(t)} \hat{e}_s, \sum_{r =1}^n a_{ri} \hat{e}_r \right\rangle \right) A \cdot \frac{\partial}{\partial a_{ij}}\nonumber\\
&=&\sum_{i,j=1}^n \left(\left\langle \nabla_{Z(t)} e_j, e_i \right\rangle
 - \left\langle \nabla_{qZ(t)} q e_j, q e_i \right\rangle \right) A \cdot \frac{\partial}{\partial a_{ij}}
\end{eqnarray}
The coefficients in the basis $A \cdot \frac{\partial}{\partial_{ij}}$ in the
sum~\eqref{sumdist} are skew symmetric, from the property of the Levi-Civita connection.
Now we can write
\begin{equation} \label{A left}
\sum_{i,j=1}^n \dot{a}_{ij} \frac{\partial}{\partial a_{ij}}
= \sum_{i<j} \left(\left\langle \nabla_{Z(t)} e_j, e_i \right\rangle
 - \left\langle \nabla_{qZ(t)} q e_j, q e_i \right\rangle \right) W_{ij}(A).
 \end{equation}
Written in a right invariant basis, we obtain
\begin{equation} \label{A right}
\sum_{i,j=1}^n \dot{a}_{ij} \frac{\partial}{\partial a_{ij}} \\
= \sum_{i<j} \left(\left\langle \nabla_{Z(t)} q^{-1} \hat{e}_j, q^{-1} \hat{e}_i \right\rangle
 -  \left\langle \nabla_{qZ(t)} \hat{e}_j, \hat{e}_i \right\rangle \right) \text{Ad}(A^{-1}) W_{ij}(A).
 \end{equation}

Similarly, (III) holds if and only if
\begin{equation} \label{B leftright}
\begin{array}{rl}
\sum\limits_{\kappa,\lambda=1}^\nu \dot{b}_{\kappa\lambda} \frac{\partial}{\partial b_{\kappa\lambda}} &
= \sum\limits_{\kappa <\lambda} \left(\left\langle \nabla_{Z(t)}^\bot \epsilon_\lambda, \epsilon_\kappa \right\rangle
 - \left\langle \nabla_{qZ(t)}^\perp p \epsilon_\lambda, p \epsilon_\kappa \right\rangle \right) W_{\kappa\lambda}(B). \\
&= \sum\limits_{\kappa <\lambda} \left(
\left\langle \nabla_{Z(t)}^\bot p^{-1} \hat{\epsilon}_\lambda, p^{-1} \hat{\epsilon}_\kappa \right\rangle
 - \left\langle \nabla_{qZ(t)}^\perp \hat{\epsilon}_\lambda, \hat{\epsilon}_\kappa \right\rangle \right)
\text{Ad}(B^{-1}) W_{\kappa\lambda}(B).
\end{array}
 \end{equation}

\begin{definition}\label{defdistr} If $X$ is a vector field on $M$, then let us define
$\mathcal{V}(X)$ and $\mathcal{V}^\bot(X)$ the vector fields on $Q \oplus P_{\iota,\widehat{\iota}}$,
such that under any local trivialization $h$ as in \eqref{trivi} and any $(q,p) \in \pi^{-1}(x)$ they satisfy
\begin{equation} \label{Vunderh}
dh \left(\mathcal{V}(X)(q,p) \right) = \sum_{i<j} \left(\left\langle \nabla_{X(x)} e_j, e_i \right\rangle
 - \left\langle \nabla_{qX(x)} q e_j, q e_i \right\rangle \right) W_{ij}(A).
 \end{equation}
\begin{equation} \label{Vperpunderh}
dh \left(\mathcal{V}^\bot(X)(q,p)\right)
= \sum_{\kappa <\lambda} \left(\left\langle \nabla_{X(x)}^\bot \epsilon_\lambda, \epsilon_\kappa \right\rangle
 - \left\langle \nabla_{qX(x)}^\perp p \epsilon_\lambda, p \epsilon_\kappa \right\rangle \right) W_{\kappa\lambda}(B).
\end{equation}
\end{definition}

Notice that if $Y(x) = X(x) = X_0 \in T_xM$, then $\mathcal{V}(Y)(q,p) = \mathcal{V}(X)(q,p)$
for every $(q,p) \in (Q \oplus P_{\iota,\widehat{\iota}})_x$. Hence, we may define
$\mathcal{V}(X_0)(q,p)$ whenever $X_0 \in T_xM$ and $(q,p) \in  (Q \oplus P_{\iota,\widehat{\iota}})_x$.
Also notice that the map $X \mapsto \mathcal{V}(X)$ is linear. The same holds for $\mathcal{V}^\bot$.

\begin{remark}
Notice that, at first glace, it may seem that all of the coefficients of $W_{ij}(A)$
and $W_{\kappa\lambda}(A)$ in~\eqref{Vunderh} and~\eqref{Vperpunderh} vanish
from conditions (II) and (III). This is not true, however. Even though, for any tangential vector field $X$
$$\frac{D}{dt} X(x(t)) = \nabla_{\dot{x}(t)} X(x(t)),$$
in general, $\nabla_{q\dot{x}(t)}q(t) e_j$ does not coincide with $\frac{D}{dt} q(t) e_j(x(t))$.
To see this, notice that
$$\frac{D}{dt} a_{sj} \hat{e}_s(\widehat{x}(t)) = \dot{a}_{sj} \hat{e}_s(\widehat{x}(t)) + a_{sj} \nabla_{\dot{\widehat{x}}(t)}\hat{e}_s(\widehat{x}(t))
= \dot{a}_{sj} \hat{e}_s(\widehat{x}(t)) + a_{sj} \nabla_{q \dot{x}(t)}\hat{e}_s(\widehat{x}(t)),$$
while
$$\nabla_{q\dot{x}(t)} a_{sj} \hat{e}_s(x(t)) = a_{sj} \nabla_{q\dot{x}(t)} \hat{e}_s(x(t)).$$
Similar relations hold for $\frac{D^\bot}{dt}$.
\end{remark}

We may now sum up our considerations that have been made in this Section in the following result.
\begin{proposition}
A curve $(q(t),p(t))$ in $Q \oplus P_{\iota,\widehat{\iota}}$ is a rolling if and only if it is a horizontal curve
with respect to the distribution $E$, defined by $$E_{(q,p)} = \left\{X_0 + qX_0 + \mathcal{V}(X_0)(q,p) + \mathcal{V}^\perp(X_0)(q,p) | \, X_0 \in T_xM \right\},\quad (q,p) \in  (Q \oplus P_{\iota,\widehat{\iota}})_x.$$
\end{proposition}
If we use the same symbol to denote the restriction of $\mathcal{V}(X)$ to $Q$, we also have
\begin{proposition}\label{propdist}
A curve $q(t)$ in $Q$ is an intrinsic rolling if and only if it is a horizontal curve
with respect to the distribution $D$, defined by $$D_q = \left\{X_0 + qX_0 + \mathcal{V}(X_0)(q) | \, X_0 \in T_xM \right\},\quad q \in Q_x.$$
\end{proposition}

\subsection{Properties of the distribution}
We present some of the properties for the distribution $E$ that basically reflects the results found in
Theorem \ref{number intrinsic} and \ref{unique extension}.
\begin{proposition}
\begin{itemize}
\item[(a)] $E$ is biinvariant under the action of $\SO(\nu)$.
\item[(b)] Let $X$ be a vector field on $M$. If for any $q \in Q_{x,\widehat{x}}$,
if $(A_0q) X(x) = q X(x)$, then
$$\mathcal{V}(X)(A_0q) = A_0 \mathcal{V}(X)(q).$$
Similarly, if $(qA_0) X(x) = q X(x)$, then
$$\mathcal{V}(X)(qA_0) = \mathcal{V}(X)(q) A_0.$$
\end{itemize}
\end{proposition}

\begin{proof}
To prove (a), we only need to show that $B_0 \cdot \mathcal{V}^\perp(X)(q,p) = \mathcal{V}^\perp(X)(q,B_0p)$
and $\mathcal{V}^\perp(X)(q,p) \cdot B_0 = \mathcal{V}^\perp(X)(q,p B_0)$ for any
$B_0 = (\tilde b_{\kappa\lambda})_{\kappa\lambda = 1}^\nu \in$.
First, let $h$ be a local trivialization as in \eqref{trivi} and $(q,p) \in (Q \oplus P_{\iota,\widehat{\iota}})_x$.
Then, from \eqref{B leftright}, to show $\mathcal{V}^\perp(X)(q,B_0p) = B_0 \cdot \mathcal{V}^\perp(X)(q,p)$,
it is sufficient to show that if $B_0 = \left(\tilde{b}_{\kappa\lambda}\right)_{\kappa,\lambda = 1}^\nu$, then
$$\sum_{\alpha, \beta, \mu, \vartheta =1}^\nu \tilde{b}_{\alpha\mu} b_{\mu \lambda} \tilde{b}_{\beta\vartheta}
b_{\vartheta\kappa} \left\langle \nabla_{qX} \hat{\epsilon}_\alpha, \hat{\epsilon}_\beta \right\rangle
= \left\langle \nabla^\perp_{qX} p \epsilon_\lambda, p\epsilon_\kappa \right\rangle.$$
Since $$\sum_{\alpha, \beta, \mu, \vartheta =1}^\nu \tilde{b}_{\alpha\mu} b_{\mu \lambda} \tilde{b}_{\beta\vartheta} b_{\vartheta\kappa} \left\langle \nabla^\perp_{qX} \hat{\epsilon}_\alpha, \hat{\epsilon}_\beta \right\rangle
= \sum_{\vartheta, \mu =1}^\nu b_{\mu \lambda}
b_{\vartheta\kappa} \left\langle \sum_{\alpha =1}^\nu \tilde{b}_{\alpha\mu} \nabla^\perp_{qX} \hat{\epsilon}_\alpha,
\sum_{\beta = 1}^\nu \tilde{b}_{\beta \vartheta} \hat{\epsilon}_\beta \right\rangle,$$
we need to show that
$$\left\langle \sum_{\alpha =1}^\nu \tilde{b}_{\alpha\mu} \nabla^\perp_{qX} \hat{\epsilon}_\alpha,
\sum_{\beta = 1}^\nu \tilde{b}_{\beta \vartheta} \hat{\epsilon}_\beta \right\rangle
= \left\langle \nabla^\perp_{qX} \hat{\epsilon}_\mu, \hat{\epsilon}_\vartheta \right\rangle.$$
To obtain this, let $\overline{Y}$ be any exstension of the vector field $qX$, to $\mathbb{R}^{n+\nu}$.
Let $\overline{\hat{\epsilon}_\mu}$ and $\overline{\hat{\epsilon}_\vartheta}$ be any exstension
of $\hat{\epsilon}_\mu$ and $\hat{\epsilon}_\vartheta$, and write
$$\overline{\hat{\epsilon}_\mu} = \sum_{i = 1}^{n+\nu} f_i(\bar{x}) \frac{\partial}{\partial \bar{x}_i}.$$
Then
$$\left\langle \sum_{\alpha =1}^\nu \tilde{b}_{\alpha\mu} \nabla^\perp_{qX} \hat{\epsilon}_\alpha,
\sum_{\vartheta = 1}^\nu \tilde{b}_{\beta \vartheta} \hat{\epsilon}_\beta \right\rangle
= \left\langle \nabla^\perp_{qX} B_0 \hat{\epsilon}_\mu,
 B_0 \hat{\epsilon}_\vartheta) \right\rangle$$
$$= \left\langle \overline{\nabla}_{\bar{Y}} B_0 \overline{\hat{\epsilon}_\mu},
 B_0 \overline{\hat{\epsilon}_\vartheta} \right\rangle = \left\langle B_0
 \sum_{i = 1}^{n+\nu} \overline{Y} f_i(\bar{x}) \frac{\partial}{\partial \bar{x}_i},
 B_0 \overline{\hat{\epsilon}_\vartheta} \right\rangle = \left\langle B_0 \overline{\nabla}_{Y}
 B_0 \overline{\hat{\epsilon}_\mu}, B_0 \overline{\hat{\epsilon}_\vartheta} \right\rangle$$
 $$= \left\langle B_0 \nabla^\bot_{qX}
 \hat{\epsilon}_\mu, B_0 \hat{\epsilon}_\vartheta \right\rangle
 = \left\langle \nabla^\bot_{qX} \hat{\epsilon}_\mu, \hat{\epsilon}_\vartheta \right\rangle.$$
We show right invariance by showing that
$$\sum_{\alpha, \beta, \mu, \vartheta = 1}^\nu \tilde{b} _{\mu \alpha} b_{\lambda\mu}
\tilde{b} _{\vartheta \beta} b_{\kappa\vartheta}
\left\langle \nabla_{Z(t)}^\bot \epsilon_\alpha, \epsilon_\beta \right\rangle
= \left\langle \nabla_{Z(t)}^\bot p^{-1}\hat{\epsilon}_\lambda, p^{-1} \hat{\epsilon}_\kappa \right\rangle,$$
in a similar way.

The proof of (b) is totally analogous to the proof of (a)
\end{proof}

\section{A controllable example: $S^n$ rolling over $\mathbb{R}^n$}
\subsection{Formulation of the rolling}
We want to illustrate the properties of the distributions, by proving that the unit sphere $S^n$ in $\mathbb{R}^{n+1}$ rolling over $\mathbb{R}^n$ is a completely controllable system, by showing that the distribution $D$ is bracket generating. This result was obtained in \cite{Zimm}, but we want to present this example here in order to illustrate the advantages of our theory.

Consider the unit sphere $S^n$ as the submanifold of the Euclidean space $\mathbb{R}^{n+1}$,
$$S^n = \left\{ (x_0,\dots, x_n) \in \mathbb{R}^{n+1} | \, x_0^2 + \cdots + x_n^2 = 1 \right\},$$
with the induced metric.

For an arbitrary point $\tilde x =(\tilde x_0, \dots, \tilde x_n)\in S^n$, at least one of the coordinates $\tilde x_0, \dots, \tilde x_n$ does not vanish. Without lost of generality, we may assume that $\tilde x_n\neq0$, and consider the neighborhood
$$U = \{(x_0, \dots, x_n) \in S^n | \, \pm x_n > 0 \} \, ,$$
where the choice of the $\pm$ sign depends on the sign of $\tilde x_n$. To simplify the notation, we define the following functions on $U$
$$s_j(x) = \sum_{r =j}^n x_r^2 \, .$$
These functions are always strictly positive on $U$, and we use them to define an orthonormal basis of $TU$. We will write simply $s_j$ instead of $s_j(x)$, since dependence of $x$ is clear from the context. Define the following vector fields on $U$
\begin{equation}\label{basisTSn}
e_j = \sqrt{\frac{s_j}{s_{j-1}}} \left(- \frac{\partial}{\partial x_{j-1}} + \frac{x_{j-1}}{s_j}
\sum_{r = j}^n x_r \frac{\partial}{\partial x_r} \right), \qquad j = 1, \dots, n.
\end{equation}
These vector fields form an orthonormal basis of the tangent space over $U$.
We set $\hat{e}_i=\dfrac{\partial}{\partial \widehat{x}_i}$ to be the standard basis of $\mathbb{R}^n$.

Before proceeding with the necessary calculations, let us state two technical Lemmas whose proofs can be found in section \ref{GammaSn} and \ref{DerGamma}.

\begin{lemma} \label{LemmaGamma}
Let $1 \leq i < j \leq n$. Then
$$\left\langle \nabla_{e_k} e_j , e_i \right\rangle = - \frac{x_{i-1} \delta_{k,j}}
{\sqrt{s_{i-1} s_i}}
= - \left\langle \nabla_{e_k} e_i , e_j \right\rangle,$$
for any $k = 1, \dots, n$.
\end{lemma}
\begin{remark}
The properties of the connection $\nabla$ have the following consequences:
\begin{itemize}
\item The compatibility of $\nabla$ with the metric and $\left\langle e_i, e_j \right\rangle = \delta_{i,j}$, imply that
$$\left\langle \nabla_{e_k} e_j , e_i \right\rangle = - \left\langle \nabla_{e_k} e_i , e_j \right\rangle.$$
In particular, $\left\langle \nabla_{e_k} e_i , e_i \right\rangle = 0.$
\item The symmetry of $\nabla$, imply that if $l < k$, then
$$[e_k, e_l] = \nabla_{e_k} e_l - \nabla_{e_l} e_k
= \sum_{i = 1}^n \left\langle \nabla_{e_k} e_l - \nabla_{e_l} e_k, e_i \right\rangle e_i
= \frac{x_{l-1}}{\sqrt{s_{l-1}s_l}} e_k.$$
\end{itemize}
\end{remark}

\begin{lemma} \label{LemmaDerGamma} For $k,l = 1, 2, \dots, n$
$$e_k\left(\frac{x_{l-1}}{\sqrt{s_{l-1} s_j}} \right)
= \left\{\begin{array}{cc} 0 & k > l \\ - \frac{1}{s_k} & k = l \\
& \\
- \frac{x_{k-1} x_{l-1}}{\sqrt{s_{k-1}s_k s_k s_{j-1}}} & k < l. \end{array} \right.$$
\end{lemma}
It is a direct consequence of the choice of the vector fields $\hat{e}_k$ that
$\nabla_{\hat{e}_k} \hat{e}_l = 0$, and $[\hat{e}_k, \hat{e}_l] = 0$ for all $k, l = 1, \dots, n$.

Consider the vector fields $X_k=e_k + q e_k + \mathcal{V}(e_k)$ which generate the distribution $\mathcal{D}$, introduced in Proposition~\ref{propdist}, restricted to $U$. In this case, we have the explicit form
$$X_k(x,\hat{x},A) = e_k(x) + \sum_{i =1}^n a_{ik} \hat{e}_i(\widehat{x}) -
\sum_{i =1}^{k-1} \frac{x_{i-1}}{\sqrt{s_{i-1}s_i}} W_{ik}(A).$$

In order to determine the commutators $[X_k, X_l]$, let us assume that $k > l$.
Then
$$[X_k,X_l] = [e_k, e_l] - \sum_{i=1}^{k-1} \sum_{j =1}^n  \frac{x_{i-1}}{\sqrt{
s_{i-1} s_i}} W_{ik} a_{jl} \hat{e}_j + \sum_{j-1}^{l-1} \sum_{i =1}^n  \frac{x_{j-1}}{\sqrt{
s_{j-1} s_j}} W_{jl} a_{ik} \hat{e}_i$$
$$- \sum_{j =1}^{l-1} e_k\left(\frac{x_{j-1}}{\sqrt{
s_{j-1} s_j}} \right) W_{jl} + \sum_{i =1}^{k-1} e_l\left(\frac{x_{i-1}}{\sqrt{
s_{i-1}s_i}} \right) W_{ik}+ \sum_{i =1}^{k-1} \sum_{j =1}^{l-1} \frac{x_{i-1} x_{j-1}}{\sqrt{
s_{i-1}s_i s_{j-1} s_{j-1}}} [W_{ik}, W_{jl}]$$
\vspace{0.5cm}
$$= \frac{x_{l-1}}{\sqrt{s_{l-1} s_l}} e_k - \sum_{i=1}^{k-1} \sum_{j =1}^n  \frac{x_{i-1}}{\sqrt{
s_{i-1} s_i}} \left(a_{ji} \delta_{k,l} - a_{jk} \delta_{i,l}\right) \hat{e}_j
+ \sum_{j-1}^{l-1} \sum_{i =1}^n  \frac{x_{j-1}}{\sqrt{s_{j-1} s_j}} \left(a_{ij} \delta_{l,k} - a_{il} \delta_{j,k}\right) \hat{e}_i$$
$$- \frac{1}{s_l} W_{lk} - \sum_{i = l+1}^{k-1} \frac{x_{i-1} x_{l-1}}
{\sqrt{s_{l-1} s_l s_{i-1} s_i}} W_{ik}
+ \sum_{i =1}^{k-1} \sum_{j =1}^{l-1} \frac{x_{i-1} x_{j-1}}{\sqrt{
s_{i-1} s_i s_{j-1} s_j}} (-\delta_{i,l} W_{jk} +\delta_{i,j} W_{lk})$$
\vspace{0.5cm}
$$= \frac{x_{l-1}}{\sqrt{s_{l-1} s_l}} \left(e_k + \sum_{j =1}^n a_{jk} \hat{e}_j
- \sum_{i = l+1}^{k-1} \frac{x_{i-1}}{\sqrt{s_{i-1} s_i}} W_{ik}
- \sum_{j =1}^{l-1} \frac{x_{j-1}}{\sqrt{s_{j-1} s_j}}W_{jk} \right)$$
$$- \frac{1}{s_l} W_{lk} + \sum_{j =1}^{l-1} \frac{x_{j-1}^2}{s_{j-1} s_j} W_{lk}$$
\vspace{0.5cm}
$$= \frac{x_{l-1}}{\sqrt{s_{l-1} s_l}} \left(
e_k + \sum_{j=1}^n a_{jk} \hat{e}_j - \sum_{i = 1}^{k-1} \frac{x_{i-1}}
{\sqrt{s_{i-1} s_i}} W_{ik}  \right)+ \left(- \frac{1}{s_l} + \frac{x_{l-1}^2}{s_{l-1} s_l} + \sum_{j =1}^{l-1} \left(\frac{1}{
s_j} -\frac{1}{s_j} \right)\right) W_{lk}$$
$$= \frac{x_{l-1}}{\sqrt{s_{l-1} s_l}} X_k- W_{lk}.$$
Define the vector fields $Y_{lk}$, for $l<k$, by
$$Y_{lk} := [X_l,X_k] + \frac{x_{l-1}}{\sqrt{s_{l-1} s_l}} X_k= W_{lk}.$$
Finally, let
$$Z_1 = [Y_{12}, X_2] = \sum_{i = 1}^n a_{i1} \hat{e}_i,$$
$$Z_k = [X_1,Y_{1k}] = \sum_{i=1}^n a_{ik} \hat{e}_i\, , \qquad k = 2,\dots, n.$$
We conclude that the entire tangent space is spanned by $\{X_k\}_{k =1}^n,  \{Y_{lk}\}_{1\leq l < k \leq n}$ and
$\{Z_k\}_{k=1}^n$. Hence, $D$ is a regular bracket generating distribution of step 3, which implies that the system of rolling $S^n$ over $\mathbb{R}^n$ is completely controllable.

\subsection{Proof of Lemma \ref{LemmaGamma}} \label{GammaSn}
The proof of this Lemma is rather technical and it consists mostly of rewriting formulas in an appropriate way. We begin with some observations.
\begin{itemize}
\item $s_{j-1} = x_{j-1}^2 + s_j$.
\item If $H$ is the Heaviside function
$$H(x) = \left\{ \begin{array}{cc} 1 & \text{when } x \geq 0 \\ 0 & \text{when } x < 0 \end{array} \right. \, ,$$
then $$\frac{\partial}{\partial x_k} s_j = 2 x_k H(k-j).$$
\item For any integer $j$,
$$H(j) = \delta_{0,j} + H(j-1).$$
\item Due to the identity
$$\left\langle \frac{\partial}{\partial x_k} , e_i \right\rangle
= \sqrt{\frac{s_i}{s_{i-1}}}\left(- \delta_{k, i-1} + \frac{x_k x_{i-1} H(k-i)}{s_i} \right),$$
we obtain
\begin{align*} \left\langle \sum_{r=k}^n x_r \frac{\partial}{\partial x_r}, e_i\right\rangle & =
\sqrt{\frac{s_i}{s_{i-1}}}\sum_{r = k}^n \left(- x_r \delta_{r, i-1} + \frac{ x_{i-1} x_r^2 H(r-i)}{s_i} \right) \\
& = x_{i-1} \sqrt{\frac{s_i}{s_{i-1}}} \left(\frac{s_{\max\{k,i\}}}{s_i}-H(i-k-1)\right) \\
& = x_{i-1} \sqrt{\frac{s_i}{s_{i-1}}} \left(\delta_{i,k} + \frac{s_k H(k-i-1)}{s_i}\right). \end{align*}
\end{itemize}
\vspace{0.5cm}
\paragraph{\it Step 1: Finding $\overline{\nabla}_{\frac{\partial}{\partial x_k}} e_j$}
We calculate
\begin{align*} \frac{\partial}{\partial x_k} \sqrt{\frac{s_j}{s_{j-1}}} & = \frac{x_kH(k-j)}{\sqrt{s_{j-1} s_j}}
- x_k H(k-j+ 1)\sqrt{\frac{s_j}{s_{j-1}^3}} \\
& = x_kH(k-j) \sqrt{\frac{s_j}{s_{j-1}}}\left(\frac{1}{s_j} - \frac{1}{s_{j-1}}\right)
- x_k \delta_{k,j-1} \sqrt{\frac{s_j}{s_{j-1}^3}} \\
& = \sqrt{\frac{s_j}{s_{j-1}}} \left(\frac{x_k x_{j-1}^2 H(k-j)}{s_{j-1} s_j}
- \frac{x_k \delta_{k,j-1}}{s_{j-1}} \right) \end{align*}
and get
$$\overline{\nabla}_{\frac{\partial}{\partial x_k}} e_j =
\left(\frac{x_k x_{j-1}^2 H(k-j)}{s_{j-1} s_j} - \frac{x_k \delta_{k,j-1}}{s_{j-1}} \right) e_j$$
$$+ \sqrt{\frac{s_j}{s_{j-1}}}\left(\frac{\delta_{k,j-1}}{s_j} \sum_{r = j}^n x_r \frac{\partial}{\partial x_r}
- \frac{2 x_{j-1} x_k H(k-j)}{s_j^2} \sum_{r = j}^n x_r \frac{\partial}{\partial x_r}
+ \frac{x_{j-1}H(k-j)}{s_j}\frac{\partial}{\partial x_k}\right)$$
\vspace{0.5cm}
$$= \left(\frac{x_k x_{j-1}^2 H(k-j)}{s_{j-1} s_j} - \frac{2 x_k H(k-j)}{s_j} - \frac{x_k \delta_{k,j-1}}{s_{j-1}} \right) e_j$$
$$+ \sqrt{\frac{s_j}{s_{j-1}}}\left(\frac{\delta_{k,j-1}}{s_j} \sum_{r = j}^n x_r \frac{\partial}{\partial x_r}
- \frac{2 x_k H(k-j)}{s_j} \frac{\partial}{\partial x_{j-1}}
+ \frac{x_{j-1}H(k-j)}{s_j}\frac{\partial}{\partial x_k}\right)$$
\vspace{0.5cm}
$$= - \left(x_ k H(k-j) \frac{s_{j-1} + s_j }{s_{j-1} s_j} + \frac{x_k \delta_{k,j-1}}{s_{j-1}} \right) e_j$$
$$+ \frac{1}{s_j} \sqrt{\frac{s_j}{s_{j-1}}}\left(\delta_{k,j-1} \sum_{r = j}^n x_r \frac{\partial}{\partial x_r}
+ H(k-j)\left(x_{j-1}\frac{\partial}{\partial x_k} - 2 x_k \frac{\partial}{\partial x_{j-1}}\right) \right)$$
\paragraph{\it Step 2: Calculating $\overline{\nabla}_{e_k} e_j$}
Using Step 1 and formula~\eqref{basisTSn}, we are able to compute
$$\overline{\nabla}_{e_k} e_j = \sqrt{\frac{s_k}{s_{k-1}}} \left(-
\overline{\nabla}_{\frac{\partial}{\partial x_{k-1}}} e_j + \frac{x_{k-1}}{s_k}
\sum_{l = k}^n x_l \overline{\nabla}_{\frac{\partial}{\partial x_l}} e_j \right)$$
\vspace{0.5cm}
$$= \sqrt{\frac{s_k}{s_{k-1}}} \left(
\left(x_ {k-1} H(k-j-1) \frac{s_{j-1} + s_j }{s_{j-1} s_j} + \frac{x_{k-1} \delta_{k,j}}{s_{j-1}} \right) e_j \right.$$
$$\left.- \frac{1}{s_j} \sqrt{\frac{s_j}{s_{j-1}}}\left(\delta_{k,j} \sum_{r = j}^n x_r \frac{\partial}{\partial x_r}
+ H(k-j-1)\left(x_{j-1}\frac{\partial}{\partial x_{k-1}} - 2 x_{k-1} \frac{\partial}{\partial x_{j-1}}\right) \right) \right.$$
$$\left. + \frac{x_{k-1}}{s_k} \left(
- \left(s_{\max\{j,k\}} \frac{s_{j-1} + s_j }{s_{j-1} s_j} + \frac{x^2_{j-1} H(j-k-1)}{s_{j-1}} \right) e_j \right. \right.$$
$$\left. \left.+ \frac{1}{s_j} \sqrt{\frac{s_j}{s_{j-1}}}\left( x_{j-1} H(j-k-1) \sum_{r = j}^n x_r \frac{\partial}{\partial x_r}
+ x_{j-1}\sum_{r = \max\{j,k\}}^n x_r \frac{\partial}{\partial x_r} - 2 s_{\max\{j,k\}} \frac{\partial}{\partial x_{j-1}}\right)  \right)\right)$$
\paragraph{\it Step 3: Obtaining $\left\langle \nabla_{e_k} e_j, e_i \right\rangle$}
We calculate it case by case,
\begin{itemize}
\item if $k = j$, then
\begin{align*} \overline{\nabla}_{e_k} e_k & = \sqrt{\frac{s_k}{s_{k-1}}} \left(
\frac{x_{k-1}}{s_{k-1}} e_k - \frac{1}{s_k} \sqrt{\frac{s_k}{s_{k-1}}} \sum_{r = k}^n x_r \frac{\partial}{\partial x_r} \right. \\
& \quad \left. + \frac{x_{k-1}}{s_k} \left(
- \frac{s_{k-1} + s_k }{s_{k-1}} e_k + \frac{1}{s_k} \sqrt{\frac{s_k}{s_{k-1}}}\left(x_{k-1}\sum_{r = k}^n x_r \frac{\partial}{\partial x_r} - 2 s_k \frac{\partial}{\partial x_{k-1}}\right)  \right)\right) \end{align*}
\vspace{0.5cm}
\begin{align*} \qquad & = \sqrt{\frac{s_k}{s_{k-1}}} \left(
\frac{x_{k-1} s_k - x_{k-1} s_k - x_{k-1} s_{k-1}}{s_{k-1} s_k} e_k - \frac{1}{s_k} \sqrt{\frac{s_k}{s_{k-1}}} \sum_{r = k}^n x_r \frac{\partial}{\partial x_r} \right. \\
& \quad \left. + \frac{x_{k-1}}{s_k} \left( e_k -
\sqrt{\frac{s_k}{s_{k-1}}} \frac{\partial}{\partial x_{k-1}} \right)\right)\\ \\
& = - \frac{1}{s_{k-1}} \sum_{r = k-1}^n x_r \frac{\partial}{\partial x_r} \end{align*}
and so
$$\left\langle \nabla_{e_k} e_k , e_i \right\rangle =
- \frac{x_{i-1}}{s_{k-1}} \sqrt{\frac{s_i}{s_{i-1}}} \left(\delta_{i,k-1} + \frac{s_{k-1} H(k-i-2)}{s_i}\right)
= - \frac{x_{i-1} H(k-i-1)}{\sqrt{s_{i-1} s_i}} $$
\item if $k <j$, then
\begin{align*} \overline{\nabla}_{e_k} e_j & = \sqrt{\frac{s_k}{s_{k-1}}} \left(\frac{x_{k-1}}{s_k} \left(
- \left(\frac{s_{j-1} + s_j }{s_{j-1}} + \frac{x^2_{j-1}}{s_{j-1}} \right) e_j \right. \right. \\
& \quad \left. \left.+ \frac{1}{s_j} \sqrt{\frac{s_j}{s_{j-1}}}\left( x_{j-1} \sum_{r = j}^n x_r \frac{\partial}{\partial x_r}
+ x_{j-1}\sum_{r = j}^n x_r \frac{\partial}{\partial x_r} - 2 s_j \frac{\partial}{\partial x_{j-1}}\right)  \right)\right) \\
& = \sqrt{\frac{s_k}{s_{k-1}}} \left(\frac{x_{k-1}}{s_k} \left(
- 2 e_j + 2e_j \right)\right) = 0 \end{align*}
\item if $j < k$, then
$$\overline{\nabla}_{e_k} e_j = \sqrt{\frac{s_k}{s_{k-1}}} \left(
\left(x_ {k-1} \frac{s_{j-1} + s_j }{s_{j-1} s_j} \right) e_j - \frac{1}{s_j} \sqrt{\frac{s_j}{s_{j-1}}}
\left(x_{j-1}\frac{\partial}{\partial x_{k-1}} - 2 x_{k-1} \frac{\partial}{\partial x_{j-1}}\right)  \right.$$
$$\left. + \frac{x_{k-1}}{s_k} \left(
- \left(s_k \frac{s_{j-1} + s_j }{s_{j-1} s_j} \right) e_j + \frac{1}{s_j} \sqrt{\frac{s_j}{s_{j-1}}}\left(
x_{j-1}\sum_{r = k}^n x_r \frac{\partial}{\partial x_r} - 2 s_k \frac{\partial}{\partial x_{j-1}}\right)  \right)\right)$$
\vspace{0.5cm}
$$= \sqrt{\frac{s_k s_j}{s_{k-1}s_{j-1}}} \left(
- \frac{x_{j-1}}{s_j} \frac{\partial}{\partial x_{k-1}} + \frac{2 x_{k-1}}{s_j} \frac{\partial}{\partial x_{j-1}}
+  \frac{x_{k-1} x_{j-1}}{s_k s_j}
\sum_{r = k}^n x_r \frac{\partial}{\partial x_r} -  \frac{2 x_{k-1}}{s_j}\frac{\partial}{\partial x_{j-1}}\right)$$
$$= \frac{x_{j-1}}{\sqrt{s_{j-1} s_j}} e_k$$
\end{itemize}
The conclusion is that all the Christoffel symbols $\Gamma_{kj}^i =
\left\langle \nabla_{e_k} e_j, e_i \right\rangle$ vanish, except for
$$\Gamma_{kj}^k = - \Gamma_{kk}^j = \frac{x_{j-1}}{\sqrt{s_{j-1} s_j}}, \qquad j < k.$$

\subsection{Proof of Lemma \ref{LemmaDerGamma}} \label{DerGamma}
We need to consider three cases. Observe that
\begin{align*} \frac{\partial}{\partial x_k} \left(\frac{x_{l-1}}{\sqrt{s_{l-1} s_l}} \right) &
=   \frac{\delta_{l-1,k}}{\sqrt{s_{l-1} s_l}}  - \frac{1}{s_{l-1}} \frac{x_{l-1}^2 \delta_{l-1,k}}{\sqrt{s_{l1} s_l}} \\
& \quad- \frac{1}{s_{l-1}}\frac{x_{l-1} x_k H(k-l)}{\sqrt{s_{l-1} s_l}} - \frac{1}{s_l}
\frac{x_{l-1} x_k H(k-l)}{\sqrt{s_{l-1} s_l}} \\ \\
& =  \frac{1}{\sqrt{s_{l-1}^3 s_l}} \left(\delta_{l-1,k} s_l - \frac{(s_l + s_{l-1})
x_{l-1} x_k H(k-l)}{s_l} \right), \\ \end{align*}
so
\begin{align*} e_k\left(\frac{x_{l-1}}{\sqrt{s_{l-1} s_l}} \right) &
= \sqrt{\frac{s_k}{s_{l-1}^3 s_l s_{k-1}}} \left(- \left(\delta_{l,k} s_l - \frac{(s_l + s_{l-1})
x_{l-1} x_{k-1} H(k-l-1)}{s_l} \right) \right.\\
 & \left.\quad + \frac{x_{k-1}}{s_k} \sum_{r = k}^n x_r
\left(\delta_{l-1,r} s_l - \frac{(s_l + s_{l-1})  x_{l-1} x_r H(r-l)}{s_l} \right) \right).\\ \\
& = \sqrt{\frac{s_k}{s_{l-1}^3 s_l s_{k-1}}} \left(- \delta_{l,k} s_l + \frac{(s_l + s_{l-1})
x_{l-1} x_{k-1} H(k-l-1)}{s_l}  \right. \\
& \quad \left. + \frac{x_{k-1}}{s_k} \left( x_{l-1} s_l H(l-1-k)  - \frac{(s_l + s_{l-1})  x_{l-1} s_{\max\{k,l\}}}{s_l} \right) \right). \end{align*}
\begin{itemize}
\item If $k = l$, then
\begin{align*} e_k\left(\frac{x_{k-1}}{\sqrt{s_{k-1} s_k}} \right) &
= \frac{1}{s_{k-1}^2} \left(- s_k - \frac{x_{k-1}}{s_k} \frac{(s_k + s_{k-1})  x_{k-1} s_k}{s_k} \right) \\
& = - \frac{1}{s_{k-1}^2}  \frac{ s^2_k +(s_k + s_{k-1})(s_{k-1} - s_k)}{s_k} = -\frac{1}{s_k}. \end{align*}
\item If $k > l$, then
$$e_k\left(\frac{x_{l-1}}{\sqrt{s_{l-1} s_l}} \right)
= \sqrt{\frac{s_k}{s_{l-1}^3 s_l s_{k-1}}} \left(\frac{(s_l + s_{l-1})
x_{l-1} x_{k-1}}{s_l}  - \frac{x_{k-1}}{s_k} \frac{(s_l + s_{l-1})  x_{l-1} s_k}{s_l} \right) = 0.$$
\item If $k < l$, then
\begin{align*} e_k\left(\frac{x_{l-1}}{\sqrt{s_{l-1} s_l}} \right) &
= \sqrt{\frac{s_k}{s_{l-1}^3 s_l s_{k-1}}} \left( \frac{x_{k-1}}{s_k} \left( x_{l-1} s_l - \frac{(s_l + s_{l-1})  x_{l-1} s_l}{s_l} \right) \right) \\
&= -\frac{x_{k-1} x_{l-1}}{\sqrt{s_{k-1} s_k s_{l-1} s_l}}. \end{align*}
\end{itemize}

\section{A non-controllable example: ${\rm SE}(3)$ rolling over $\mathbb{R}^6$}

\subsection{Calculation of the dimension of the orbits}
Let ${\rm SE}(3)$ be the group of orientation preserving isometries of $\mathbb{R}^3$. We consider the case of ${\rm SE}(3)$, endowed with a left invariant metric defined later, rolling over its tangent space at the identity $T_1\SE(3)=\mathfrak{se}(3)$, with metric obtained by restricting the left invariant metric on ${\rm SE}(3)$ to the identity. Our goal is to determine whether any two points in the configuration space can be joined by a curve tangent to the distribution presented in Definition~\ref{defdistr}. This problem is equivalent to the controllability of the system, that is, we want to obtain any configuration by rolling without twisting or slipping, from a given an initial configuration.

We give ${\rm SE}(3)$ coordinates as follows. For any $x \in {\rm SE}(3)$ there exist $C=(c_{ij})\in{\rm SO}(3)$ and $r=(r_1,r_2,r_3)\in \mathbb{R}^3$, such that $x=(C,r)$ acts via
$$x(y) = C y + r, \qquad \mbox{for all } y \in \mathbb{R}^3.$$

The tangent space of ${\rm SE}(3)$ at $x=(C,r)$ is spanned by the left invariant vector fields
\begin{equation} \begin{array}{rl}
e_1 = Y_1&=\frac{1}{\sqrt{2}} \left( C \cdot \frac{\partial}{\partial c_{12}} - C \cdot \frac{\partial}{\partial c_{21}}\right)\\ \\
&=\frac{1}{\sqrt{2}} \sum_{j=1}^3 \left(c_{j1} \frac{\partial}{\partial c_{j2}} - c_{j2} \frac{\partial}{\partial c_{j1}}\right)\end{array} \end{equation} \vspace{0.3cm}
\begin{equation} \begin{array}{rl}
e_2 = Y_2&= \frac{1}{\sqrt{2}} \left( C \cdot \frac{\partial}{\partial c_{13}} - C \cdot \frac{\partial}{\partial c_{31}}\right)\\ \\
&=\frac{1}{\sqrt{2}} \sum_{j=1}^3 \left(c_{j1} \frac{\partial}{\partial c_{j3}} - c_{j3} \frac{\partial}{\partial c_{j1}}\right)
\end{array} \end{equation} \vspace{0.3cm}
\begin{equation} \begin{array}{rl}
e_3 = Y_3&=\frac{1}{\sqrt{2}} \left( C \cdot \frac{\partial}{\partial c_{23}} - C \cdot \frac{\partial}{\partial c_{32}}\right)\\ \\
&=\frac{1}{\sqrt{2}} \sum_{j=1}^3 \left(c_{j2} \frac{\partial}{\partial c_{j3}} - c_{j3} \frac{\partial}{\partial c_{j2}}\right)
\end{array} \end{equation} \vspace{0.3cm}
\begin{align}
e_{k+3} = X_k &=C \cdot \frac{\partial}{\partial r_k} = \sum_{j=1}^3 c_{jk} \frac{\partial}{\partial r_j} & 
k=1,2,3.
\end{align}

Define a left invariant metric on ${\rm SE}(3)$ by declaring the vectors $e_1,\ldots,e_6$ to form an orthonormal basis. The mapping
$$\sum_{j=1}^6 \widehat{x}_j e_j(1) \mapsto (\widehat{x}_1, \widehat{x}_2, \widehat{x}_3, \widehat{x}_4, \widehat{x}_5, \widehat{x}_6) \in \mathbb{R}^6,$$
permits to identify $\mathfrak{se}(3)$ endowed with the induced metric, with $\mathbb{R}^6$ with the Euclidean metric.
We write $\hat{e}_k=\frac{\partial}{\partial \widehat{x}_k}$ on $\mathbb{R}^6$ and try to see how the intrinsic rollings of ${\rm SE}(3)$ on $\mathbb{R}^6$ behave. Note that $Q = {\rm SE}(3) \times \mathbb{R}^6 \times {\rm SO}(6)$, because both manifolds ${\rm SE}(3)$ and $\mathbb{R}^6$ are Lie groups, so their tangent bundles are trivial, and $\dim Q=27$.

Let us denote by $\nabla$ the Levi-Civita connection on ${\rm SE}(3)$ or $\mathbb{R}^6$ with respect to the corresponding Riemannian metrics defined above. The covariant derivatives $\nabla_{e_i} e_j$ are nonzero only in the following cases
\begin{eqnarray*}
\nabla_{Y_1} Y_2&=&- \nabla_{Y_2} Y_1 = - \frac{1}{2\sqrt{2}} Y_3\\
\nabla_{Y_1} Y_3&=&- \nabla_{Y_3} Y_1 = \frac{1}{2\sqrt{2}} Y_2\\
\nabla_{Y_2} Y_3&=&- \nabla_{Y_3} Y_2 = - \frac{1}{2\sqrt{2}} Y_1\\
\nabla_{Y_1} X_k&=&\frac{1}{\sqrt{2}}\left(\delta_{2,k} X_1 - \delta_{1,k} X_2 \right)\\
\nabla_{Y_2} X_k&=&\frac{1}{\sqrt{2}}\left(\delta_{3,k} X_1 - \delta_{1,k} X_3 \right)\\
\nabla_{Y_3} X_k&=&\frac{1}{\sqrt{2}}\left(\delta_{3,k} X_2 - \delta_{2,k} X_3 \right),
\end{eqnarray*}
where $\delta_{i,j}$ denotes the Kronecker symbol. On the other hand, it is well-known that $\nabla_{\hat{e}_i} \hat{e}_j= 0$ for any $i,j$. Proposition~\ref{propdist} and Definition~\ref{defdistr} show that the distribution $D$ over $Q$ is spanned by
\begin{equation}\label{distSE3}
\begin{array}{rcl}
Z_1&=&Y_1 + q Y_1 + \frac{1}{2\sqrt{2}} W_{23} + \frac{1}{\sqrt{2}} W_{45},\\
&&\\
Z_2&=&Y_2 + q Y_2 - \frac{1}{2\sqrt{2}} W_{13} + \frac{1}{\sqrt{2}} W_{46},\\
&&\\
Z_3&=&Y_3 + q Y_3 + \frac{1}{2\sqrt{2}} W_{12} + \frac{1}{\sqrt{2}} W_{56},\\
K_1&=&X_1 + q X_1,\\
K_2&=&X_2 + q X_2,\\
K_3&=&X_3 + q X_3.
\end{array}
\end{equation}

In order to determine the controllability of rolling ${\rm SE}(3)$ over $\mathbb{R}^6$, we employ the Orbit Theorem \cite{Hermann,Suss}.
In the case of $D$, defined by the vector fields~\eqref{distSE3} straightforward calculations yield
that the flag associated to $D$ is on the form
\begin{equation}\label{commSE3}
\begin{array}{rcl}
D^2&=&D \oplus {\rm span}\,\{W_{12}, W_{13}, W_{23}\},\\
D^3&=&D^2 \oplus {\rm span}\{qY_1, q Y_2, q Y_3\},\\
D^4&=&D^3,
\end{array}
\end{equation}
and so $\dim D^2 = 9$, $\dim D^k=12$ for all $k\geq3$ and the step of $D$ is 3.

Let $(x_0, \widehat{x}_0, A_0)$ be an arbitrary point in $Q$, and let $\mathcal{O}_{(x_0, \widehat{x}_0, A_0)}$ denote the subset of all points in $Q$ which are connected to $(x_0, \widehat{x}_0, A_0)$ by an intrinsic rolling. The Orbit Theorem asserts that, at each point, $D^3$ is contained in
the tangent space of the orbits. However, since we know that $D^3$ has a local basis, we
have the stronger result of
$$T_{(x,\widehat{x},A)}\mathcal{O}_{(x_0, \widehat{x}_0, A_0)}=D^3_{(x,\widehat{x},A)},$$
holding for all $(x,\widehat{x},A)\in\mathcal{O}_{(x_0, \widehat{x}_0, A_0)}$.

It follows from~\eqref{commSE3} that $\mathcal{O}_{(x_0, \widehat{x}_0, A_0)}$ has dimension 12. Since $\mathcal{O}_{(x_0, \widehat{x}_0, A_0)}$ is not the entire $Q$, we conclude that the system is not controllable.

We end this Section with a concrete example of an intrinsic rolling $q(t) = (x(t), \widehat{x}(t), A(t))$, where
$$x(0) = \text{id}_{\mathbb{R}^3}, \quad \widehat{x}(0) = 0, \quad A(0) = {\bf 1}.$$
Define the curve $x:[0,\tau] \to {\rm SE}(3)$ by
\begin{equation} \label{SEex}
x(t) y = \left(\begin{array}{ccc}
\cos \theta(t) & \sin \theta(t) & 0 \\ -\sin \theta(t) & \cos \theta(t) & 0 \\
0 & 0 & 1  \end{array} \right) \left(\begin{array}{c} y_1 \\ y_2 \\ y_3 \end{array} \right)
+ \left(\begin{array}{c} 0 \\ 0 \\ \psi(t) \end{array}  \right) \, ,
\end{equation}
where $\theta(t)$ and $\psi(t)$ are absolutely continuous functions with $\theta(0) = \psi(0) = 0$. Then $\dot{x} = \sqrt{2}\, \dot{\theta}(t) Y_1 + \dot{\psi}(t) X_3$ for almost every $t$, and the rolling has the form $\dot{q} = \sqrt{2}\, \dot{\theta}(t) Z_1 + \dot{\psi}(t) K_3$, or equivalently
\begin{eqnarray}
\dot x(t)&=&\sqrt{2}\, \dot{\theta}(t)Y_1 + \dot{\psi}(t) X_3\label{xdotSE3}\\
\dot{\widehat x}(t)&=&\sqrt{2}\, \dot{\theta}(t)\,q Y_1+ \dot{\psi}(t)\, q X_3\label{xhatdotSE3}\\
\dot A(t)&=&\dot{\theta}(t)\,\left(\frac{1}{2} W_{23}(A) + W_{45}(A)\right)\label{isomdotSE3}
\end{eqnarray}
for almost every $t$. It follows from equation~\eqref{xhatdotSE3} that
$$\widehat{x}(t) = \left(\begin{array}{c} \sqrt{2}\, \theta(t) \\ 0 \\ 0 \\ 0 \\ 0 \\ \psi(t) \end{array} \right).$$
Equation~\eqref{isomdotSE3} can be written as
$$\dot A(t)=A\cdot\left(\frac{\dot{\theta}(t)}{2} W_{23}(1) +\dot{\theta}(t)\, W_{45}(1)\right),$$
which can be solved by exponentiating to obtain
$$A(t) =
\left(\begin{array}{cccccc} 1 & 0 & 0 & 0 & 0 & 0 \\
0 & \cos\left(\frac{\theta(t)}2\right) & \sin\left(\frac{\theta(t)}2\right) & 0 & 0 & 0 \\
0 & - \sin\left(\frac{\theta(t)}2\right) & \cos\left(\frac{\theta(t)}2\right) & 0 & 0 & 0 \\
0 & 0 & 0 & \cos \theta(t) & \sin \theta(t) & 0 \\
0 & 0 & 0 & -\sin \theta(t) & \cos \theta(t) & 0 \\
0 & 0 & 0 & 0 & 0 & 1 \end{array} \right).
$$

\subsection{Imbedding of $\SE(n)$ into Euclidean space}\label{imbedSE(n)}
Since it is less obvious how to extend an intrinsic rolling of $\SE(3)$ on $\se(3)$ to an extrinsic rolling in ambient space, we describe an isometric imbedding of $\SE(n)$ into the Euclidean space $\real^{(n+1)^2}$. Identify an element $\overline{C} \in \real^{(n+1)^2}$ with the matrix
$$\overline{C} = \left(\begin{array}{ccc} \overline{c}_{11} & \cdots & \overline{c}_{1, n+1} \\
\vdots & \ddots & \vdots \\ \overline{c}_{n+1, 1} & \cdots & \overline{c}_{n+1, n+1} \end{array} \right) \, .$$
Define the inner product on $\real^{(n+1)^2}$ by
$$\left\langle \overline{C}_1 , \overline{C}_2 \right\rangle = \tr\left(\left(\overline{C}_1\right)^t \overline{C}_2\right).$$
Note that since
$$\left\langle \overline{C} , \overline{C} \right\rangle = \sum_{i,j = 1}^{n+1} |\overline{c}_{ij}|^2 \, ,$$
the metric $\langle\cdot,\cdot\rangle$ coincides with the Euclidean metric. From this we get that $\left\{\tfrac{\partial}{\partial \overline{c}_{ij}} \right\}_{i,j=1}^{n+1}$ is an orthonormal basis for the tangent bundle $T\real^{(n+1)^2}$ with respect to $\langle\cdot,\cdot\rangle$.

We define the imbedding of $\SE(3)$ into $\mathbb{R}^{(n+1)^2}$ by
$$\begin{array}{rrcl} \iota: & \SE(n) & \to & \real^{(n+1)^2} \\
& x = (C,r) & \mapsto & \overline{C} = \left( \begin{array}{cc} C & r \\ 0 & 1 \end{array} \right) \end{array}$$
This mapping is in fact an isometry of $\SE(n)$ onto its image. To see this, notice that the metrics coincide at the identity, and that the metric of $\real^{(n+1)^2}$, restricted to $\im \iota$, is left invariant under the action of $\SE(n)$. Hence, the metrics on $\SE(n)$ and $\im \iota$ coincide, and $\iota$ defines an isometric imbedding.

\subsection{Extrinsic rolling}
We will use the imbedding from Subsection~\ref{imbedSE(n)} to construct an extrinsic rolling of $SE(3)$ over $\mathfrak{se}(3)$ in $\mathbb{R}^{16}$. We use $\partial_{ij}$ to denote $\tfrac{\partial}{\partial \overline{c}_{ij}}$. For the sake of clarity, we denote by $M$ the image of $SE(3)$ by $\iota$. Then the vector fields spanning $TM$ are
$$e_1 = Y_1 = \frac{1}{\sqrt{2}} \sum_{i=1}^3 \left(\overline{c}_{i1} \partial_{i2} -
\overline{c}_{i2} \partial_{i1} \right),$$
$$e_2 = Y_2 = \frac{1}{\sqrt{2}} \sum_{i=1}^3 \left(\overline{c}_{i1} \partial_{i3}
- \overline{c}_{i3} \partial_{i1} \right),$$
$$e_3 = Y_3 = \frac{1}{\sqrt{2}} \sum_{i=1}^3 \left(\overline{c}_{i2} \partial_{i3}
- \overline{c}_{i3} \partial_{i2} \right),$$
$$e_{3+k} = X_k = \sum_{j=1}^3 \overline{c}_{ik} \partial_{i4}, \qquad k = 1, 2, 3,$$
where we suppressed $d\iota$ in the notation. We introduce an othonormal basis of $TM^\perp$
$$\Upsilon_1 = \frac{1}{\sqrt{2}} \sum_{j=1}^3 \left(\overline{c}_{j1} \partial_{j2}
+ \overline{c}_{j2} \partial_{j1} \right),$$
$$\Upsilon_2 = \frac{1}{\sqrt{2}} \sum_{j=1}^3 \left(\overline{c}_{j1} \partial_{j3}
+ \overline{c}_{j3} \partial_{j1} \right),$$
$$\Upsilon_3 = \frac{1}{\sqrt{2}} \sum_{j=1}^3 \left(\overline{c}_{j2} \partial_{j3}
+ \overline{c}_{j3} \partial_{j2} \right),$$
$$\Psi_\lambda =
\sum_{j=1}^3 \overline{c}_{j\lambda} \partial_{j\lambda}, \qquad \lambda =1, 2, 3$$
$$\Xi_\lambda = \partial_{4\mu}, \qquad \mu = 1, 2, 3, 4$$

We denote by $\widehat{M}$ the image of $\real^6$ into $\real^{16}$ by the imbedding
$$(\widehat{x}_1 , \widehat{x}_2, \widehat{x}_3, \widehat{x}_4, \widehat{x}_5, \widehat{x}_6)
\stackrel{\widehat{\iota}}{\mapsto}
\left(\begin{array}{cccc} 0 & \frac{1}{\sqrt{2}} \widehat{x}_1 & \frac{1}{\sqrt{2}} \widehat{x}_2
&  \widehat{x}_4 \\ - \frac{1}{\sqrt{2}} \widehat{x}_1 & 0 & \frac{1}{\sqrt{2}} \widehat{x}_3
& \widehat{x}_5 \\ - \frac{1}{\sqrt{2}} \widehat{x}_2 & - \frac{1}{\sqrt{2}} \widehat{x}_3 & 0
& \widehat{x}_6 \\ 0 & 0 & 0 & 0 \end{array} \right).$$

We have the following orthonormal basis of $T\widehat{M}$,
$$\hat{e}_1 = \frac{1}{\sqrt{2}}(\partial_{12} - \partial_{21}),$$
$$\hat{e}_2 = \frac{1}{\sqrt{2}}(\partial_{13} - \partial_{31}),$$
$$\hat{e}_3 = \frac{1}{\sqrt{2}}(\partial_{23} - \partial_{32}),$$
$$\hat{e}_{3+k} = \partial_{k4} \qquad k = 1, 2 ,3$$
while the vector fields spanning $T\widehat{M}^\bot$,
$$\hat{\epsilon}_1 = \frac{1}{\sqrt{2}}(\partial_{12} + \partial_{21}),$$
$$\hat{\epsilon}_2 = \frac{1}{\sqrt{2}}(\partial_{13} + \partial_{31}),$$
$$\hat{\epsilon}_3 = \frac{1}{\sqrt{2}}(\partial_{23} + \partial_{32}),$$
$$\hat{\epsilon}_{3+\kappa} = \partial_{\kappa\kappa}, \qquad \kappa = 1, 2 ,3,$$
$$\hat{\epsilon}_{6+\kappa} = \partial_{4\kappa} \qquad \kappa = 1, 2 ,3, 4.$$

In order to extend an intrinsic rolling $q(t)$ with $\pi(q(t)) = (x(t) , \widehat{x}(t))$, we will find an orthonormal frame of normal parallel vector fields along $x(t)$ and $\widehat{x}(t)$. Along $\widehat{x}(t)$, we may use the restriction of $\left\{ \hat{\epsilon}_\kappa \right\}_{\kappa = 1}^{10}$. For the curve $x(t)$ the answer is more complicated.

We first study the value of $\nabla^\bot$ for different choices of vector fields.
\begin{enumerate}
\item $\nabla^\bot_X \Xi_\lambda = 0$, for any tangential vector field $X$.
\item $\nabla^\perp_{X_k} \Upsilon = 0$, for any normal vector field $\Upsilon$.
\item Otherwise
$$\begin{array}{|c|c|c|c|c|c|c|}\hline
& \Upsilon_1 & \Upsilon_2 & \Upsilon_3 & \Psi_1 & \Psi_2  & \Psi_3 \\ \hline
\nabla^\bot_{Y_1} & \frac{1}{2} \left(\Psi_1 - \Psi_2\right)
& - \frac{1}{2\sqrt{2}} \Upsilon_3 & \frac{1}{2\sqrt{2}} \Upsilon_2
& - \frac{1}{2} \Upsilon_1 & \frac{1}{2} \Upsilon_1 & 0 \\ \hline
\nabla^\bot_{Y_2} & - \frac{1}{2\sqrt{2}} \Upsilon_3 & \frac{1}{2} \left(\Psi_1 - \Psi_3\right) &
\frac{1}{2\sqrt{2}} \Upsilon_1 &
- \frac{1}{2} \Upsilon_2 & 0 & \frac{1}{2} \Upsilon_2 \\ \hline
\nabla^\bot_{Y_3} & - \frac{1}{2\sqrt{2}} \Upsilon_2 & \frac{1}{2\sqrt{2}} \Upsilon_1
&  \frac{1}{2} \left(\Psi_2 - \Psi_3\right) &
0 & - \frac{1}{2} \Upsilon_3 & \frac{1}{2}\Upsilon_3 \\ \hline
\end{array}\,.$$
\end{enumerate}
We use the above relation to construct an extrinsic rolling. We will illustrate this by considering the curve~\eqref{SEex}.

Since $\dot{x}(t) = \sqrt{2} \, \dot{\theta}(t) Y_1(x(t)) + \dot{\psi}(t) X_3(x(t))$, the vector field
$$\Psi(t) = \sum_{\lambda =1}^3 \left(\upsilon_\lambda(t) \Upsilon_\lambda(x(t)) + \upsilon_{3 + \lambda}(t) \Psi_\lambda(x(t)) \right) \, ,$$
is normal parallel along $x(t)$ if
\begin{align*}\left(\dot{\upsilon}_1 - \frac{\dot{\theta}}{\sqrt{2}}\left(\upsilon_4 - \upsilon_5\right)\right) & \Upsilon_1
+ \left(\dot{\upsilon}_2 + \frac{\dot{\theta}}2 \upsilon_3\right) \Upsilon_2
+ \left(\dot{\upsilon}_3 - \frac{\dot{\theta}}2 \upsilon_2\right) \Upsilon_3 \\
&+ \left(\dot{\upsilon}_4 + \frac{\dot{\theta}}{\sqrt{2}} \upsilon_1\right) \Psi_1 
+ \left(\dot{\upsilon}_5 - \frac{\dot{\theta}}{\sqrt{2}} \upsilon_1\right) \Psi_2
+ \upsilon_6 \Psi_3 = 0.\end{align*}
Hence we define a parallel orthonormal frame along $x(t)$ by
$$\epsilon_1(t) = \cos \theta \Upsilon_1(x(t)) - \frac{1}{\sqrt{2}} \sin \theta \Psi_1(x(t))
+ \frac{1}{\sqrt{2}} \sin \theta \Psi_2(x(t)) \, ,$$
$$\epsilon_2(t) = \cos\left( \frac{\theta}{2} \right) \Upsilon_2(x(t)) + \sin\left( \frac{\theta}{2} \right)
\Upsilon_3(x(t)) \, ,$$
$$\epsilon_3(t) = - \sin\left( \frac{\theta}{2} \right) \Upsilon_2(x(t)) + \cos\left( \frac{\theta}{2} \right)
\Upsilon_3(x(t)) \, ,$$
$$\epsilon_4(t) = \frac{1}{\sqrt{2}} \sin \theta \Upsilon_1(x(t))
+ \frac{\cos \theta + 1}{2} \Psi_1(x(t)) + \frac{1 - \cos \theta}{2} \Psi_2(x(t)) \, ,$$
$$\epsilon_5(t) = - \frac{1}{\sqrt{2}} \sin \theta \Upsilon_1(x(t))
+ \frac{1- \cos \theta}{2} \Psi_1(x(t)) + \frac{1 + \cos \theta}{2} \Psi_2(x(t)) \, ,$$
$$\epsilon_6(t) = \Psi_3(x(t)) \, ,$$
$$\epsilon_{6 + \lambda}(t) = \Xi_\lambda(x(t))\, , \qquad \lambda = 1, 2, 3, 4.$$
Thus $p(t)$ is represented by a constant matrix in the bases $\{ \epsilon_\lambda(t) \}_{\lambda =1}^{10}$ and $\{ \hat{\epsilon}_\kappa(t) \}_{\kappa =1}^{10}$. Let us choose $p(t)$ to be the identity in these bases, since this is the configuration given by the imbedding.

The curve $g(t)=(q(t),p(t))$ in $\Isom^+(\real^{16})$ is given by
$$g(t) \overline{x} = \overline{A} \overline{x} + \overline{r}(t),$$
where
$$\tiny \overline{A}(t) = \left(
\begin{array}{cccccccccc|c}
  \cos^2\tfrac{\theta}{2} & -\frac{\sin \theta}{2} & 0 & 0 & \frac{\sin \theta}{2} & \frac{\cos \theta -1 }{2} & 0 & 0 & 0 & 0 & \\
  \frac{\sin \theta}{2} & \cos^2\tfrac{\theta}{2} & 0 & 0 & \sin^2 \tfrac{\theta}{2} & \frac{\sin \theta}{2} & 0 & 0 & 0 & 0 & \\
  0 & 0 & \cos\tfrac{\theta}{2} & 0 & 0 & 0 & \sin\tfrac{\theta}{2} & 0 & 0 & 0 & \\
  0 & 0 & 0 & 1 & 0 & 0 & 0 & 0 & 0 & 0 &  \\
  -\frac{\sin \theta}{2} & \sin^2\tfrac{\theta}{2} & 0 & 0 & \cos^2\tfrac{\theta}{2} & -\frac{\sin \theta}{2} & 0 & 0 & 0 & 0 & \mbox{\large \bf 0}_{10\times6} \\
  \frac{\cos \theta -1}{2} & -\frac{\sin \theta}{2} & 0 & 0 & \frac{\sin \theta}{2} & \cos^2 \tfrac{\theta}{2} & 0 & 0 & 0 & 0 & \\
  0 & 0 & -\sin \tfrac{\theta}{2} & 0 & 0 & 0 & \cos\tfrac{\theta}{2} & 0 & 0 & 0 & \\
  0 & 0 & 0 & 0 & 0 & 0 & 0 & 1 & 0 & 0 & \\
  0 & 0 & 0 & 0 & 0 & 0 & 0 & 0 & \cos\tfrac{\theta}{2} & -\sin\tfrac{\theta}{2} & \\
  0 & 0 & 0 & 0 & 0 & 0 & 0 & 0 & \sin\tfrac{\theta}{2} &\cos\tfrac{\theta}{2} & \\ 
    &  &  &  &  &  &  &  &  &  & \\\hline
    &  &  &  &  &  &  &  &  &  & \\
    &  &  &  & \mbox{\large \bf 0}_{6\times10} &  &  &  &  &  & \mbox{\large \bf 1}_6\\
\end{array} \right) \, ,$$
and $\displaystyle{\overline{r}(t) = \left( -1, \frac{\theta}{\sqrt{2}}, 0, 0, \frac{\theta}{\sqrt{2}}, -1,
 0, 0, 0, 0, -1, 0, 0, 0, 0, 0 \right)^t}$. Here, ${\bf 0}_{m\times n}$ denotes the zero matrix of size $m\times n$ and ${\bf 1}_6$ is the identity matrix of size $6\times6$.

\end{document}